\documentclass[12pt]{amsart}
\usepackage{bbm}
\usepackage{mathrsfs}
\usepackage{}
\usepackage{amsfonts}
\usepackage{amssymb}
\usepackage{amsmath}
\usepackage{float}
\allowdisplaybreaks[4]

\newcommand{\uS}{\mathbb{S}^{n}}

\newtheorem{theorem}{Theorem}[section]
\newtheorem{lemma}[theorem]{Lemma}
\newtheorem{proposition}[theorem]{Proposition}

 \theoremstyle{definition}
\newtheorem{definition}[theorem]{Definition}

\theoremstyle{remark}
\newtheorem{remark}[theorem]{Remark}

\numberwithin{equation}{section}

\begin{document}

\title[$L_p$ dual Christoffel-Minkowski type problem]
{The $L_p$ dual Christoffel-Minkowski type problem for a class of Hessian quotient equations}

\author{Shasha Luo}
\address{Faculty of Mathematics and Statistics, Hubei Key Laboratory of Applied Mathematics, Hubei University,  Wuhan 430062, P.R. China}
\email{202421104011295@stu.hubu.edu.cn}

\author{Jiabao Gong}
\address{Faculty of Mathematics and Statistics, Hubei Key Laboratory of Applied Mathematics, Hubei University,  Wuhan 430062, P.R. China}
\email{202321104011284@stu.hubu.edu.cn}

\author{Qiang Tu$^{\ast}$}
\address{Faculty of Mathematics and Statistics, Hubei Key Laboratory of Applied Mathematics, Hubei University,  Wuhan 430062, P.R. China}
\email{qiangtu@hubu.edu.cn}

\keywords{$L_p$ dual Minkowski type problem; full rank theorem; a priori estimates}

\subjclass[2010]{Primary 35J15; Secondary 35B45.}
\thanks{$\ast$ Corresponding author}

\begin{abstract}
 In this paper, we investigate an $L_p$ dual Christoffel-Minkowski type problem for the Hessian quotient operator $\frac{\sigma_{k}(\Lambda)}{\sigma_{l}(\Lambda)}$, where the operator $\Lambda$ has been widely studied in the literature. Exploiting the recently discovered ``inverse convexity'' property of this class of operators, we establish a full rank theorem under suitable structural assumptions. Together with a priori estimates, this result enables us to prove the existence and uniqueness of strictly spherically convex solutions to the above $L_p$ dual Christoffel-Minkowski type problem.

\end{abstract}

\maketitle
%%%%%%%%%%%%%%%%%%%%%%%%%%%%%%%%%%%
\section{Introduction}

%The classical $L_p$ Brunn-Minkowski theory is the classical core of the geometry of convex bodies.   In [46],  Lutwak, Yang and Zhang [46] constructed a unifying family of geometric measures, the $(p, q)$-th dual curvature measures, which include $L_p$ surface area measure and the $L_p$ integral curvature in  $L_p$ Brunn-Minkowski theory.  This new family of geometric measures has significantly expanded the classical theory, leading to the development of the dual $L_p$ Brunn-Minkowski theory.

The $L_p$ dual Minkowski problem is one of the central problems in the $L_p$  dual Brunn-Minkowski theory. As an extension of the $L_p$ dual Minkowski problem, the $L_p$ dual Christoffel-Minkowski problem involves finding a convex body whose support function satisfies the following fully nonlinear partial differential equation
\begin{equation}\label{12-19}
\begin{split}
\sigma_{k}(\nabla^2 u +u I)=u^{p-1}(u^{2}+|\nabla u|^{2})^{\frac{k+1-q}{2}}\varphi, \quad \mbox{on}~\uS,
\end{split}
\end{equation}
where $\sigma_{k}$ is the $k$-th elementary symmetric polynomial, $u$ is the support function of the convex body, $\nabla^2 u$ is the Hessian matrix of $u$ with respect to a local orthonormal frame on $\mathbb{S}^{n}$, $I$ is the identity matrix and $\varphi$ is a smooth positive function on $\mathbb{S}^{n}$.

When $p=1$ and $q=k+1$, equation \eqref{12-19} corresponds to the Christoffel-Minkowski problem of prescribing the $k$-th area measure, which has attracted much attention. For $k=1$, equation \eqref{12-19} is the classical
Christoffel problem,  which was first posed by Christoffel \cite{Cho-65} and subsequently further developed by Berg \cite{Ber-69} and Firey \cite{Fir-67, Fir-68}. In the case $k = n$, equation \eqref{12-19} corresponds to the classical Minkowski problem, which has been extensively studied and solved in works by Minkowski \cite{Min-97, Min-03}, Alexandrov \cite{Ale-38, Ale-39},  Nirenberg \cite{Nir-53},  Cheng-Yau \cite{Ch-Yau-76}, Pogorelov \cite{Pog-78} and others. In the intermediate case $1 < k < n$, the Christoffel-Minkowski problem has been extensively examined, with significant contributions from Firey \cite{Fir-70}, Guan-Ma \cite{GM-2003}, Guan-Lin-Ma \cite{GLM-2006}, Guan-Ma-Zhou \cite{GMZ-2006}, and Bryan-Ivaki-Scheuer \cite{BIS-2023}.

When $q=k+1$, equation \eqref{12-19} reduces to the $L_p$ Christoffel-Minkowski problem. For $k=n$, equation \eqref{12-19} corresponds to the $Lp$ Minkowski problem, which was first  introduced by Lutwak \cite{Lut-1993} and has since been extensively studied. One can refer to \cite{Bianchi2019,Boreczky2013,Bereczky2017,Chen2017,Chou2006,Lu2013,Lutwak2004,Zhu2014,Zhu2015a,Zhu2015b} . For $1 < k < n$, the $L_p$ Christoffel-Minkowski problem has been studied by Hu-Ma-Shen \cite{HMS-04} for the case $p > k + 1$ (see also Ivaki \cite{IVAKI-2019} and Sheng-Yi \cite{ShC-2020}), while Guan-Xia \cite{GX-18} studied the case $1 < p < k + 1$.

Note that the $L_{p}$ dual Christoffel-Minkowski problem \eqref{12-19} includes several of the above problems as special cases. For $k=n$, equation \eqref{12-19} corresponds to the $L_p$ dual Minkowski problem of prescribing
the $(p, q)$-th dual curvature measure, first introduced by Lutwak-Yang-Zhang \cite{LYZ-2018-1}. The existence of solutions has been extensively investigated. See, e.g., Huang-Zhao \cite{HZ-2018}, B\"or\"oczky-Fodor \cite{BF-2019}, Chen-Huang-Zhao \cite{CHZ-2019} and Chen-Li \cite{Chen-Li-2021}. For the general case $1 \leq k < n$, several results have been established. See, e.g., Ding-Li \cite{DL-2023}, Chen-Tu-Xiang \cite{CTX-25} and Cabezas-Moren-Hu\cite{CH-25}.

Inspired by the above pioneering works, we naturally consider a much wider class of $L_p$ Minkowski type problems. Specifically, let $u$ be a smooth positive function on  $\uS$, and let $\lambda=(\lambda_{1},\lambda_{2},\cdot\cdot\cdot,\lambda_{n})$ be the eigenvalues of the symmetric matrix $\nabla^{2}u+uI$. Given an integer $\mathscr{P}$ with $1\leq \mathscr{P}\leq n-1$, set
\begin{equation*}
\begin{split}
\mathfrak{I}=\{(i_{1},\cdot\cdot\cdot,i_{\mathscr{P}})|1\leq i_{1}<\cdot\cdot\cdot<i_{\mathscr{P}}\leq n\},
\end{split}
\end{equation*}
and for $\mathbf{I}=(i_{1},\cdot\cdot\cdot,i_{\mathscr{P}})\in\mathfrak{I}$, define
\begin{equation*}
\begin{split}
\Lambda_{\mathbf{I}}(\lambda)=\sum_{i_{j}\in \mathbf{I}}\lambda_{i_j}=\lambda_{i_{1}}+\cdot\cdot\cdot+\lambda_{i_{\mathscr{P}}}.
\end{split}
\end{equation*}
For convenience, fix an order for the elements in $\mathfrak{I}$:
\begin{equation*}
\begin{split}
\mathbf{I}_{1},\cdot\cdot\cdot,\mathbf{I}_{N},\quad \mbox{where} \quad N=C_{n}^{\mathscr{P}}=\frac{n!}{\mathscr{P}!(n-\mathscr{P})!}.
\end{split}
\end{equation*}
Then the $\mathscr{P}$-eigenvalues of $\nabla^{2}u+uI$ are defined as
\begin{equation}\label{8987}
\begin{split}
\Lambda(\nabla^{2}u+uI)=(\Lambda_{\mathbf{I}_{1}}(\lambda),\cdot\cdot\cdot,\Lambda_{\mathbf{I}_{N}}(\lambda)).
\end{split}
\end{equation}
In this paper, we consider the following Hessian quotient type equation
\begin{equation}\label{1123}
\begin{split}
\frac{\sigma_{k}(\Lambda(\nabla^{2}u+uI))}{\sigma_{l}(\Lambda(\nabla^{2}u+uI))}=u^{p-1}(u^{2}+|\nabla u|^{2})^{\frac{k+1-q}{2}}\varphi,  \quad \mbox{on}~\uS,
\end{split}
\end{equation}
where $0\leq l <k \leq N$ are integers and $p, q$ are constants. It is clear that the equation \eqref{1123} corresponds to the $L_p$ dual Christoffel-Minkowski problem \eqref{12-19} when $l=0$ and $\mathscr{P}=1$. When $\mathscr{P}=n-1$ and $q=k+1$, equation \eqref{1123} becomes the following mixed Hessian quotient equation
\begin{equation*}\label{1123-n-1}
\begin{split}
\frac{\sigma_{k}(\eta)}{\sigma_{l}(\eta)}=u^{p-1}\varphi,
\end{split}
\end{equation*}
where $\eta=(\eta_1, \eta_2, \cdots, \eta_n)$ with $\eta_i= \sum_{j \neq i} \lambda_j$, which was introduced and  solved by Chen-Xu \cite{CX-22, CX-23} for $p\geq k-l+1$ or $-(k-l)<p-1<0$.

It is worth noting that the operators $\Lambda$ in \eqref{8987} are quite natural and have appeared frequently in the literature. A $C^2$ function $u$ is called $\mathscr{P}$-plurisubharmonic if the $\mathscr{P}$-eigenvalues of the  matrix $\nabla^2 u$ are nonnegative. Such functions were first introduced by Harvey-Lawson \cite{HL-13}. Dinew \cite{D-23} then studied a Monge-Amp\`ere type equation
\begin{equation*}\label{1123-p-n}
\begin{split}
\mathcal{M}_{\mathscr{P}}(\nabla^2 u):=\det(\Lambda(\nabla^2 u))=f,
\end{split}
\end{equation*}
 in the class of $\mathscr{P}$-plurisubharmonic functions and established a priori estimates and a Liouville-type theorem. The corresponding curvature equations were further investigated by Dong \cite{Dong-23, Dong-24}. When $\mathscr{P}=n-1$, the operators $\Lambda$ also arise naturally in connection with the Gauduchon conjecture in complex geometry. More precisely, Tosatti-Weinkove \cite{TW-17} studied the complex Monge-Amp\`ere equation
 \begin{equation}\label{32161}
\begin{split}
\det \left(\left(\sum_{i=1}^n \frac{\partial^2 u}{\partial z_i \partial \overline{z}_i}\right)\delta_{ij}-\frac{\partial^2 u}{\partial z_i \partial \overline{z}_j}\right)=f,\end{split}
\end{equation}
 for $(n-1)$-plurisubharmonic functions on  compact K\"ahler manifold. Equation \eqref{32161} is also related to the Gauduchon conjecture, which was solved by Sz\'ekelyhidi-Tosatti-Weinkove in \cite{STW-17}, where the existence of Gauduchon metrics with prescribed volume form on any compact complex manifold was established. One can consult \cite{G-84, GN-21}  for further related works. For some other recent developments of fully nonlinear equations related to the operators $\Lambda$ see \cite{CDH, CTX-20, CTX3, CTX-21, CJ-21, GZT-25, Z-24} and the references therein.

  %the complex geometry. For instance, Tosatti-Weinkove the  consider Monge-Amp\`ere type euquation for $(n-1)$-plurisubharmonic functions on a compact Kahler Manifold. In , Székelyhidi-Tosatti-Weinkove  solves the Gauduchon  conjecture by finding Gauduchon metrics with prescribed volume form on  any compact complex manifold.

%Note that the operators $\frac{\sigma_{k}(\Lambda)}{\sigma_{l}(\Lambda)}$ are very natural and appeared often in many literatures. If $A=D^2 u$, Harvey and Lawson introduced $\mathscr{P}$-plurisubharmonic  function related to the operator $\Lambda$, then Dinew study a Monge-Amp\`ere type equation ($k=n, l=0$) in the class of $\mathscr{P}$-plurisubharmonic functions and establish a priori estimates and a Liouville type theorem.

%It is well known that the support function $u$ of a strictly convex body must be convex on $\uS$. So, for any $x\in \uS$, the Hessian matrix $\Delta^2 u+uI>0$. A natural problem is whether we can find convex solutions to \eqref{1123}. Here we try to discuss this problem. With the help of a new phenomenon of “inverse convexity” with respect to the operator $\frac{\sigma_{k}(\Lambda)}{\sigma_{l}(\Lambda)}$, we can set up the Full Rank Theorem. Then, we show the existence of the convexsolutions by the method of continuity.

The support function $u$ of a strictly convex body is strictly spherically convex, i.e. the Hessian matrix  $\nabla^2 u+uI$ is positive definite. A natural question arises as to whether strictly spherically convex solutions to equation \eqref{1123} exist. In this work, we investigate this problem. By leveraging a newly observed phenomenon of ``inverse convexity'' associated with the operator $\frac{\sigma_{k}(\Lambda)}{\sigma_{l}(\Lambda)}$ (see Proposition \ref{L1}), we establish the Full Rank Theorem. Subsequently, the existence of strictly spherically convex solutions is demonstrated via the method of continuity.

\begin {theorem}\label{L6}
Let $0\leq l< k\leq C_{n-1}^{\mathscr{P}-1}$, $n\geq 2$, and $\mathscr{P}>1$. Let $\varphi(x)$ be a smooth positive function on $\mathbb{S}^{n}$. If $p, q$ and $\varphi$ satisfy one of the following conditions:
\begin{enumerate}
\item[(\romannumeral1)] $p\geq 1> q-l$, $(\varphi ^{-\frac{1}{k-l+p-1}})_{ii}+\varphi ^{-\frac{1}{k-l+p-1}}\geq0$;

\item[(\romannumeral2)] $p>q-l\geq1$, $q\leq k+1$, $(\varphi ^{-\frac{1}{k-l+p-1}})_{ii}+\varphi ^{-\frac{1}{k-l+p-1}}\geq0$;

\item[(\romannumeral3)] $p>q-l$, $q>k+1$, $(\varphi ^{-\frac{1}{k-l+p-1}})_{ii}+\frac{2k+p-l-1}{k-l+p-1}\varphi ^{-\frac{1}{k-l+p-1}}\geq0$;
\end{enumerate}
then equation \eqref{1123} admits a unique positive strictly spherically convex solution $u$.
\end {theorem}

\begin{remark}
It is worth noting that the Hessian quotient operator $\frac{\sigma_{k}(\Lambda)}{\sigma_{l}(\Lambda)}$ when $k\leq C_{n-1}^{\mathscr{P}-1}$ turns out to satisfy a very strong property (see Proposition \ref{P3}).  This property enables us to deal with the third order derivatives in  $C^2$ estimates.
\end{remark}

Then we consider the equation \eqref{1123} in the homogeneous case $p=q-l$.

\begin {theorem}\label{L7}
Let $0\leq l< k\leq C_{n-1}^{\mathscr{P}-1}$, $n\geq 2$, $\mathscr{P}>1$, and $p=q-l>1$. Let $\varphi(x)$ be a smooth positive function on $\mathbb{S}^{n}$. If $q$ and $\varphi$ satisfy one of the following conditions:
\begin{enumerate}
\item[(\romannumeral1)] $l+1<q \leq k+1$, $(\varphi ^{-\frac{1}{k-l+p-1}})_{ii}+\varphi ^{-\frac{1}{k-l+p-1}}\geq0$;

\item[(\romannumeral2)] $q>k+1$, $(\varphi ^{-\frac{1}{k-l+p-1}})_{ii}+\frac{2k+p-l-1}{k-l+p-1}\varphi ^{-\frac{1}{k-l+p-1}}\geq0$;
\end{enumerate}
then there exists a unique positive constant $\gamma$ such that the equation
\begin{equation*}\label{32158}
\begin{split}
\frac{\sigma_{k}(\Lambda(\nabla^{2}u+uI))}{\sigma_{l}(\Lambda(\nabla^{2}u+uI))}=u^{p-1}(u^{2}+|\nabla u|^{2})^{\frac{k+1-q}{2}}\gamma\varphi(x),\,\,\,\mbox{on}\,\,\,\mathbb{S}^{n},
\end{split}
\end{equation*}
admits a unique positive strictly spherically convex solution $u$ up to dilations.
\end {theorem}

In particular, we consider the equation \eqref{1123} in the special case $q=k+1$.
\begin{theorem}\label{L8}
Let $0\leq l< k\leq N$, $n\geq 2$ and $\mathscr{P}>1$. Let $\varphi(x)$ be a smooth positive function on $\mathbb{S}^{n}$. Suppose that $(\varphi ^{-\frac{1}{k-l+p-1}})_{ii}+\varphi ^{-\frac{1}{k-l+p-1}}\geq0$. Then
\begin{enumerate}
\item[(\romannumeral1)] if $p>k-l+1$, then there exists a unique positive strictly spherically convex solution $u$ to the equation
\begin{equation}\label{32150}
\begin{split}
\frac{\sigma_{k}(\Lambda(\nabla^{2}u+uI))}{\sigma_{l}(\Lambda(\nabla^{2}u+uI))}=u^{p-1}\varphi(x),\,\,\,\mbox{on}\,\,\,\mathbb{S}^{n};
\end{split}
\end{equation}
\item[(\romannumeral2)] if $p=k-l+1$, then there exists a unique positive constant $\gamma$ such that the equation
\begin{equation}\label{32157}
\begin{split}
\frac{\sigma_{k}(\Lambda(\nabla^{2}u+uI))}{\sigma_{l}(\Lambda(\nabla^{2}u+uI))}=u^{p-1}\gamma\varphi(x),\,\,\,\mbox{on}\,\,\,\mathbb{S}^{n},
\end{split}
\end{equation}
\end{enumerate}
 admits a unique positive strictly spherically convex solution $u$ up to dilations.
\end{theorem}

\begin{remark}
Theorem \ref{L8} may be viewed as an extension of Chen-Xu \cite{CX-22}, who established the existence of strictly spherically convex solutions to equations \eqref{32150} and \eqref{32157} in the case $\mathscr{P}=n-1$.
\end{remark}

The rest of the paper is organized as follows. Section 2 contains preliminaries. In Section 3, we establish a full rank theorem for equation \eqref{1123}, which is used to preserve convexity along the continuity path. Section 4 is devoted to the nonhomogeneous case $p>q-l$, where we derive the a priori estimates and prove Theorem \ref{L6}. Section 5 treats the homogeneous case and proves Theorem \ref{L7}. Theorem \ref{L8} is obtained as a special case.

%%%%%%%%%%%%%%%%%%%%%%%%%%%%%%%%%%%%%%%%%%%%%%%%%%%%%%%%%%%%%%%%%%%%%%%%%%%%%%%%%%%%%%%%%%%%%%%%%%%%%%%%%%%%%%%%%%%%%%%%%%%%%%%%%%%%%%%%%%%%%%%%%%%%%%%%%%%%%%%%%%%%
\section{Preliminaries}
%%%%%%%%%%%%%%%%%%%%%%%%%%%%%%%%%%%%%%%%%%%%%%%%%%%%%%%%%%%%%%%%%%%%%%%%%%%%%%%%%%%%%%%%%%%%%%%%%%%%%%%%%%%%%%%%%%%%%%%%%%%%%%%%%%%%%%%%%%%%%%%%%%%%%%%%%%%%%%%%%%%%

In this section, we recall the definition and some basic properties of elementary symmetric functions, which can be found in \cite{Lieberman1996} and \cite{CTX-25}.
\begin{definition}
Let $\lambda=(\lambda_1,\cdots,\lambda_n)\in\mathbb{R}^n$, then we recall the definitions of the elementary symmetric functions for $1\leq k\leq n$
\begin{equation*}
\sigma_k(\lambda)= \sum _{1 \le i_1 < i_2 <\cdots<i_k\leq
n}\lambda_{i_1}\lambda_{i_2}\cdots\lambda_{i_k}.
\end{equation*}
For convenience, let $\sigma_{0}=1$ and $\sigma_{k}=0$ for $k>n$ or $k<0$.
\end{definition}
\begin{definition}
Let $1\leq k\leq n$ and $\Gamma_k$ be a cone in $\mathbb{R}^n$ determined by
$$\Gamma_k  = \{ \lambda  \in \mathbb{R}^n :\sigma _i (\lambda ) >
0,~\forall~ 1 \le i \le k\}.$$
%A function $u\in C^2(\mathbb{S}^n)$ is called admissible if $W=(u_{ij}+u\delta_{ij})$ is in $\Gamma_k$ for each $x\in \mathbb{S}^n$.
\end{definition}

Denote $\sigma_{k-1}(\lambda|i)=\frac{\partial
\sigma_k}{\partial \lambda_i}$ and
$\sigma_{k-2}(\lambda|ij)=\frac{\partial^2 \sigma_k}{\partial
\lambda_i\partial \lambda_j}$, then we list some properties of
$\sigma_k$ which will be used later.

\begin{proposition}
Let $\lambda=(\lambda_1,\cdots,\lambda_n)\in\mathbb{R}^n$ and $1\leq
k\leq n$. Then we have
\begin{enumerate}
\item[(\romannumeral1)] $\Gamma_1\supset \Gamma_2\supset \cdot\cdot\cdot\supset
\Gamma_n$;

\item[(\romannumeral2)] $\sigma_{k-1}(\lambda|i)>0$ for $\lambda \in \Gamma_k$ and
$1\leq i\leq n$;

\item[(\romannumeral3)] $\sigma_k(\lambda)=\sigma_k(\lambda|i)
+\lambda_i\sigma_{k-1}(\lambda|i)$ for $1\leq i\leq n$;

\item[(\romannumeral4)] If $\lambda \in \Gamma_{k}$ and $0\leq l<k$, then
$$\sum_{i=1}^{n}\frac{\partial[\frac{\sigma_{k}}{\sigma_{l}}]^{\frac{1}{k-l}}}
{\partial \lambda_i}\geq [\frac{C^k_n}{C^l_n}]^{\frac{1}{k-l}};$$

\item[(\romannumeral5)] $\Big[\frac{\sigma_k}{\sigma_l}\Big]^{\frac{1}{k-l}}$ are
concave in $\Gamma_k$ for $0\leq l<k$;

\item[(\romannumeral6)] If $\lambda_1\geq \lambda_2\geq \cdot\cdot\cdot\geq \lambda_n$,
then $\sigma_{k-1}(\lambda|1)\leq \sigma_{k-1}(\lambda|2)\leq
\cdot\cdot\cdot\leq \sigma_{k-1}(\lambda|n)$ for $\lambda \in
\Gamma_k$;

\item[(\romannumeral7)]
$\sum_{i=1}^{n}\sigma_{k-1}(\lambda|i)=(n-k+1)\sigma_{k-1}(\lambda)$.
\end{enumerate}
\end{proposition}
\begin{proof}
All the properties are well known. For example, see \cite{Lieberman1996} for proofs of (\romannumeral1), (\romannumeral2), (\romannumeral3), (\romannumeral6) and (\romannumeral7), see Lemma 2.2.19 in \cite{Gerhardt2006} for the proof of (\romannumeral4), see \cite{CNS-85} and \cite{Lieberman1996} for the proof of (\romannumeral5).
\end{proof}
We then have the following basic properties, whose proofs can be found in \cite{CC} and Lemma 2.3-2.5 of \cite{Z-24}.
\begin{proposition}
Let $A=\{a_{ij}\}$ be an $n\times n$ symmetric matrix, $\lambda(A)=(\lambda_1,\lambda_2,\cdots,\lambda_n)$ be the eigenvalues of the symmetric matrix $A$. Suppose that $A=\{a_{ij}\}$ is diagonal and $\lambda_i=a_{ii}$, then we have
\begin{eqnarray*}
\frac{\partial\lambda_i}{\partial a_{ij}}=
\begin{cases}
1,~~~&i=j,\\
0,~~~& \mbox{otherwise},
\end{cases}
\end{eqnarray*}
\begin{eqnarray*}
\frac{\partial^2\lambda_i}{\partial a_{ij}\partial a_{pq}}=
\begin{cases}
\frac{1}{\lambda_i-\lambda_p},~& i=q, j=p, i\neq p,\\
~~0,& \mbox{otherwise}.
\end{cases}
\end{eqnarray*}
\end{proposition}
\begin{proposition}\label{p1}
Suppose $A=\{a_{ij}\}$ is diagonal and $m(1\leq m\leq n)$ is a positive integer, then
\begin{eqnarray*}
\frac{\partial \sigma_m(A)}{\partial a_{ij}}=
  \begin{cases}
  \sigma_{m-1}(A|i), ~&i=j,\\
  0, ~~~&otherwise,
  \end{cases}
\end{eqnarray*}
\begin{eqnarray*}
\frac{\partial^2\sigma_m(A)}{\partial a_{ij}\partial a_{pq}}=
\begin{cases}
\sigma_{m-2}(A|ip), ~~~&i=j,p=q,i\neq p,\\
-\sigma_{m-2}(A|ip), ~~~&i=q,j=p,i\neq j,\\
0,~~~&\mbox{otherwise}.
\end{cases}
\end{eqnarray*}
\end{proposition}
\begin{proposition}\label{NM}
For $\lambda \in \Gamma_m$ and $m > l \geq 0$, $ r > s \geq 0$, $m
\geq r$, $l \geq s$, we have
\begin{align}
\Bigg[\frac{{\sigma _m (\lambda )}/{C_n^m }}{{\sigma _l (\lambda
)}/{C_n^l }}\Bigg]^{\frac{1}{m-l}} \le \Bigg[\frac{{\sigma _r
(\lambda )}/{C_n^r }}{{\sigma _s (\lambda )}/{C_n^s
}}\Bigg]^{\frac{1}{r-s}}. \notag
\end{align}
\end{proposition}
\begin{proposition}
Let $A=\{a_{ij}\}$ be an $n\times n$ symmetric matrix, $\lambda=(\lambda_{1},\lambda_{2},\cdot\cdot\cdot,\lambda_{n})$ be the eigenvalues of the symmetric matrix $A$ with $\lambda_{1}\geq\cdot\cdot\cdot\geq\lambda_{n}$ and $\Lambda=(\Lambda_{\mathbf{I}_{1}},\cdot\cdot\cdot,\Lambda_{\mathbf{I}_{N}})\in\Gamma_{k}$ with $\Lambda_{\mathbf{I}_{1}}\geq\cdot\cdot\cdot\geq\Lambda_{\mathbf{I}_{N}}$, where
$$\Lambda_{\mathbf{I}_{1}}=\lambda_{1}+\cdot\cdot\cdot+\lambda_{\mathscr{P}},$$
and
$$\Lambda_{\mathbf{I}_{N}}=\lambda_{n-\mathscr{P}+1}+\cdot\cdot\cdot+\lambda_{n},$$
we have
\begin {equation*}
\frac{\partial[\frac{\sigma_{k}(\Lambda)}{\sigma_{l}(\Lambda)}]}{\partial \Lambda_{\mathbf{I}_{1}}}\leq\cdot\cdot\cdot\leq\frac{\partial[\frac{\sigma_{k}(\Lambda)}{\sigma_{l}(\Lambda)}]}{\partial \Lambda_{\mathbf{I}_{N}}}, \quad for\quad0\leq l<k\leq N,
\end {equation*}
and
\begin {equation*}
\frac{\partial[\frac{\sigma_{k}(\Lambda)}{\sigma_{l}(\Lambda)}]}{\partial \lambda_{1}}\leq\cdot\cdot\cdot\leq\frac{\partial[\frac{\sigma_{k}(\Lambda)}{\sigma_{l}(\Lambda)}]}{\partial \lambda_{n}}, \quad for\quad0\leq l<k\leq N.
\end {equation*}
\end{proposition}
\begin{proposition}\label{P3}
Let $A=\{a_{ij}\}$ be an $n\times n$ symmetric matrix, $\lambda=(\lambda_{1},\lambda_{2},\cdot\cdot\cdot,\lambda_{n})$ be the eigenvalues of the symmetric matric $A$ with $\lambda_{1}\geq\cdot\cdot\cdot\geq\lambda_{n}$ and $\Lambda=(\Lambda_{\mathbf{I}_{1}},\cdot\cdot\cdot,\Lambda_{\mathbf{I}_{N}})\in\Gamma_{k}$ with $\Lambda_{\mathbf{I}_{1}}\geq\cdot\cdot\cdot\geq\Lambda_{\mathbf{I}_{N}}$. If $0\leq l<k\leq C_{n-1}^{\mathscr{P}-1}$, then there exists a positive constant
$c(n,k,l,\mathscr{P})$ such that
\begin {equation*}
\frac{\partial[\frac{\sigma_{k}(\Lambda)}{\sigma_{l}(\Lambda)}]^{\frac{1}{k-l}}}{\partial \lambda_{i}}\geq c(n,k,l,\mathscr{P})\sum_{i=1}^{n}\frac{\partial[\frac{\sigma_{k}(\Lambda)}{\sigma_{l}(\Lambda)}]^{\frac{1}{k-l}}}{\partial \lambda_{i}},\quad for\quad1\leq i\leq n.
\end {equation*}
\end{proposition}

\begin{proposition}\label{P4}
Let $\lambda=(\lambda_{1},\lambda_{2},\cdot\cdot\cdot,\lambda_{n})$ be the eigenvalues of the symmetric matrix $A$ and $\Lambda=(\Lambda_{\mathbf{I}_{1}},\cdot\cdot\cdot,\Lambda_{\mathbf{I}_{N}})$ with $0\leq l<k \leq N$, then $[\frac{\sigma_{k}(\Lambda)}{\sigma_{l}(\Lambda)}]^{\frac{1}{k-l}}$ are concave with respect to $\lambda$ and
\begin {equation*}
\sum_{i=1}^{n}\frac{\partial[\frac{\sigma_{k}(\Lambda)}{\sigma_{l}(\Lambda)}]^{\frac{1}{k-l}}}{\partial \lambda_{i}}\geq \mathscr{P}(\frac{C_{N}^{k}}{C_{N}^{l}})^{\frac{1}{k-l}}>0.
\end {equation*}
\end{proposition}

Fix $ \mathscr{P} \in \{1, \cdots, n\}$ and $n\geqslant 2$, we use the standard notation for ordered multi-indices
$$ \mathfrak{J}(\mathscr{P},n):=\{\mathbf{I}=(i_{1},\cdots,i_{\mathscr{P}})\mid i_{s}~integers~and~1\leqslant i_{1}<\cdots<i_{\mathscr{P}}\leqslant n\}. $$
For convenience, fix an order for the elements in $\mathfrak{J}(\mathscr{P},n)$:
$$\mathbf{I}_1, \mathbf{I}_2, \cdots, \mathbf{I}_N, \quad \mbox{with}~N:=C_n^\mathscr{P}= \frac{n!}{\mathscr{P}!(n-\mathscr{P})!}.$$
Set $\mathfrak{J}(0,n)=\{0\}$ and $|\mathbf{I}|=\mathscr{P}$ if $\mathbf{I}\in \mathfrak{J}(\mathscr{P},n)$. For $\mathbf{I}\in \mathfrak{J}(\mathscr{P},n)$,
\begin{enumerate}
\item[(\romannumeral1)]  $\overline{\mathbf{I}}$ is the element in $\mathfrak{J}(n - \mathscr{P},n)$ which complements $\mathbf{I}$ in $\{1,2,\cdots,n\}$ in the natural increasing order.
\item [(\romannumeral2)] $\mathbf{I}-i$ means the multi-index of length $\mathscr{P}-1$ obtained by removing $i$ from $\mathbf{I}$ for any $i \in \mathbf{I}$.
\item [(\romannumeral3)] $\mathbf{I}+j$ means the multi-index of length $\mathscr{P}+1$ obtained by adding $j$ to $\mathbf{I}$ for any $j \notin \mathbf{I}$.
\item [(\romannumeral4)] $\sigma(\mathbf{I},\mathbf{J})$ is the sign of the permutation which reorders $(\mathbf{I},\mathbf{J})$ in the natural increasing order for any
multi-index $\mathbf{J}$ with $\mathbf{I}\cap \mathbf{J}=\emptyset$. In particular set $\sigma(\overline{0},0):=1$.
\end{enumerate}

\begin{definition}
Let $\mathscr{P} \in \{1, \cdots, n\}$, the $(\mathscr{P}, k)$-cone is defined by
$$\mathcal {P}_{\mathscr{P}, k}=\{\lambda=(\lambda_1, \lambda_2, \cdots, \lambda_n)\in \mathbb{R}^n\mid \sigma_j(\Lambda) >0, \forall 1\leq j\leq k\},$$
where $\Lambda=(\Lambda_{\mathbf{I}_1}, \Lambda_{\mathbf{I}_2}, \cdots, \Lambda_{\mathbf{I}_N}) \in \mathbb{R}^N$,  and
$$\Lambda_\mathbf{I}=\lambda_{i_1}+\lambda_{i_2}+\cdots+\lambda_{i_\mathscr{P}},$$
for any $\mathbf{I}=(i_1, i_2, \cdots, i_\mathscr{P})\in \mathfrak{J}(\mathscr{P},n)$.
\end{definition}
\begin{remark}
It is worth emphasizing that $\lambda \in \mathcal {P}_{\mathscr{P}, k}$ if and only if $\Lambda \in \Gamma_k$ in $\mathbb{R}^N$.
\end{remark}
\begin{definition}
Let $u\in C^{2}(\mathbb{S}^{n})$, we call $u$ an admissible solution of \eqref{1123} if
$$\lambda(\nabla^{2}u+uI) \in \mathcal {P}_{\mathscr{P}, k},\quad \mbox{for any}\,\, x\in \mathbb{S}^{n}.$$
\end{definition}
\begin{definition}
Let $\mathscr{P} \in \{1, \cdots, n\}$ and  $A=\{a_{ij}\}$ be a symmetric matrix,  the linear derivation of $A$ on $\Lambda^\mathscr{P}\mathbb{R}^n$
as the linear map is defined by
\begin{eqnarray*}
\begin{aligned}
\mathcal{D}_A:&~~~~~~~\Lambda^\mathscr{P}\mathbb{R}^n &&\longrightarrow \Lambda^\mathscr{P}\mathbb{R}^n \\
&v_1\wedge\cdots\wedge v_\mathscr{P} &&\longrightarrow \sum_{i=1}^{\mathscr{P}} v_1 \wedge \cdots \wedge v_{i-1}\wedge (Av_i)\wedge v_{i+1}  \cdots\wedge v_\mathscr{P}.
\end{aligned}
\end{eqnarray*}
\end{definition}
 Fix an orthonormal basis $(e_{1},\cdot\cdot\cdot,e_{n})$ of $\mathbb{R}^n$ and the corresponding basis $\{e_{\mathbf{I}}\}_{\mathbf{I}\in\mathfrak{J}(\mathscr{P},n)}$ of $\Lambda^\mathscr{P}\mathbb{R}^n$, where $e_\mathbf{I}:=e_{i_1}\wedge\cdots\wedge e_{i_\mathscr{P}}$ for any $\mathbf{I}=(i_1,\cdots i_\mathscr{P}) \in \mathfrak{J}(\mathscr{P},n)$. Obviously, the eigenvalues of    $\mathcal{D}_A$ can be written as
 $\Lambda_{\mathbf{I}}=(\Lambda_{\mathbf{I}_1}, \Lambda_{\mathbf{I}_2}, \cdots, \Lambda_{\mathbf{I}_N})$ if  $\lambda=(\lambda_1,\lambda_2,\cdots,\lambda_n)$ are the eigenvalues of symmetric matrix $A$.

It is worth emphasizing that given any orthonormal basis of $\mathbb{R}^n$, $\mathcal{D}_A$ has a matrix representation with respect to the induced basis whose components are linear combinations of the entries of $A$.

\begin{proposition}\label{p2}
Let $A=\{a_{ij}\}$ be a symmetric matrix, the corresponding matrix $W:=\{ W_{\mathbf{I}\mathbf{J}}\}_{\mathbf{I},\mathbf{J}\in \mathfrak{J}(\mathscr{P},n)}$ of linear derivation $\mathcal{D}_A$
in the canonical basis $\{ e_{\mathbf{I}_1}, e_{\mathbf{I}_2}, \cdots, e_{\mathbf{I}_N}\}$ reads
\begin{eqnarray*}
W_{\mathbf{I}\mathbf{J}}=
\begin{cases}
\sum_{i\in \mathbf{I}}a_{ii}, ~~~&\mathbf{I}=\mathbf{J},\\
\sigma(i,\mathbf{I}-i)\sigma(j,\mathbf{J}-j)a_{ij}, ~~~&\mathbf{I}=i+\mathbf{K},\,\mathbf{J}=j+\mathbf{K},\,|\mathbf{K}|=\mathscr{P}-1,i\neq j,\\
0,~~~&\mbox{otherwise},
\end{cases}
\end{eqnarray*}
and hence we have
\begin{eqnarray*}
\frac{\partial W_{\mathbf{I}\mathbf{J}}}{\partial a_{ij}}=
\begin{cases}
1, ~~~&i=j,\mathbf{I}=\mathbf{J},\,i\in \mathbf{I},\\
\sigma(i,\mathbf{I}-i)\sigma(j,\mathbf{J}-j), ~~~&\mathbf{I}=i+\mathbf{J},\,\mathbf{J}=j+\mathbf{K},\,|\mathbf{K}|=\mathscr{P}-1,\,i\neq j,\\
0,~~~&\mbox{otherwise}.
\end{cases}
\end{eqnarray*}
\end{proposition}

\begin{proof}
For any $ \mathbf{I}\in \mathfrak{J}$ with $\mathbf{I}=(i_1,i_2,\cdots,i_\mathscr{P})$, we have
\begin{align*}
\mathcal{D}_A(e_\mathbf{I})=&\sum _{i_s\in \mathbf{I}}e_{i_1}\wedge e_{i_2}\wedge\cdots\wedge e_{i_{s-1}}\wedge Ae_{i_s}\wedge e_{i_{s+1}}\wedge\cdots\wedge e_{i_\mathscr{P}}\\
=&\sum _{i_s\in \mathbf{I}}e_{i_1}\wedge e_{i_2}\wedge\cdots\wedge e_{i_{s-1}}\wedge \sum_{m=1}^na_{i_sm}e_m \wedge e_{i_{s+1}}\wedge\cdots\wedge e_{i_\mathscr{P}}\\
=&\sum _{i_s\in \mathbf{I}}e_{i_1}\wedge e_{i_2}\wedge\cdots\wedge e_{i_{s-1}}\wedge \left(\sum_{m\notin\{i_1,i_2,\cdots,i_{s-1},i_{s+1},\cdots,i_\mathscr{P}\}}a_{i_sm}e_m \right)\wedge e_{i_{s+1}}\wedge\cdots\wedge e_{i_\mathscr{P}}\\
=&\sum _{i_s\in \mathbf{I}}e_{i_1}\wedge e_{i_2}\wedge\cdots\wedge e_{i_{s-1}}\wedge (a_{i_si_s}e_{i_s}+\sum_{m\notin \mathbf{I}}a_{i_sm}e_m) \wedge e_{i_{s+1}}\wedge\cdots\wedge e_{i_\mathscr{P}}\\
=&\sum _{i_s\in \mathbf{I}}a_{i_si_s}e_{\mathbf{I}}+\sum _{\{i_s|i_s\in \mathbf{I}\}}\sum _{\{m|m\notin \mathbf{I}\}}a_{i_sm}e_{i_1}\wedge e_{i_2}\wedge\cdots\wedge e_{i_{s-1}}\wedge e_{m}\wedge e_{i_{s+1}}\wedge\cdots\wedge e_{i_\mathscr{P}}\\
=&\sum _{i\in \mathbf{I}}a_{ii}e_{\mathbf{I}}+\sum _{i\in \mathbf{I}}\sum _{j\notin \mathbf{I}}a_{ij}\sigma (i,\mathbf{I}-i)\sigma(j,\mathbf{I}-i)e_{\mathbf{I}-i+j}.
\end{align*}

Then the formulas follows from direct calculations.
\end{proof}

\begin{proposition}\label{L1}
Let $\lambda=(\lambda_1,\cdots,\lambda_n)\in \Gamma_{n}$ be the eigenvalues of the matrix $\{a_{ij}\}$,
then $-\left[\frac{\sigma_{k}(W)}{\sigma_{l}(W)}\right]^{-\frac{1}{k-l}}$ is ``inverse convexity'' with respect to the matrix
$\{a_{ij}\}$, that is, for any $n \times n$ symmetric matrix $\{\xi_{ij}\}$,
\begin{eqnarray*}\label{in-1}
\sum_{i,j,r,s=1}^{n}\left(\frac{\partial^{2}\big(-\left[\frac{\sigma_{k}(W)}{\sigma_{l}(W)}\right]^{-\frac{1}{k-l}}\big)}{\partial a_{ij}\partial a_{rs}}+2\frac{\partial\big(-\left[\frac{\sigma_{k}(W)}{\sigma_{l}(W)}\right]^{-\frac{1}{k-l}}\big)}{\partial a_{ir}}a^{js}\right)\xi_{ij}\xi_{rs}\geq 0,
\end{eqnarray*}
where $\{a^{ij}\}=\{a_{ij}\}^{-1}$.
\end{proposition}
\begin{proof}
First, it is well known that $-\left[\frac{\sigma_{k}(W)}{\sigma_{l}(W)}\right]^{-\frac{1}{k-l}}$ is ``inverse convexity'' with respect to $W$. Therefore, for any $N \times N$ symmetric matrix $\{\xi_{\mathbf{I}\mathbf{J}}\}$
\begin{align}\label{in-2}
\sum_{\mathbf{I},\mathbf{J},\mathbf{R},\mathbf{S}\in\mathfrak{J}}\left(\frac{\partial^{2}\big(-\left[\frac{\sigma_{k}(W)}{\sigma_{l}(W)}\right]^{-\frac{1}{k-l}}\big)}{\partial W_{\mathbf{I}\mathbf{J}}\partial W_{\mathbf{R}\mathbf{S}}}+2\frac{\partial\big(-\left[\frac{\sigma_{k}(W)}{\sigma_{l}(W)}\right]^{-\frac{1}{k-l}}\big)}{\partial W_{\mathbf{I}\mathbf{R}}}W^{\mathbf{J}\mathbf{S}}\right)\xi_{\mathbf{I}\mathbf{J}}\xi_{\mathbf{R}\mathbf{S}}\geq0,
\end{align}
where $\{W^{\mathbf{J}\mathbf{S}}\}=\{W_{\mathbf{J}\mathbf{S}}\}^{-1}$. For convenience, we introduce the following notation
\begin{eqnarray*}
\begin{aligned}
\widetilde{G}^{\mathbf{I}\mathbf{J}}=\frac{\partial\big(-\left[\frac{\sigma_{k}(W)}{\sigma_{l}(W)}\right]^{-\frac{1}{k-l}}\big)}{\partial W_{\mathbf{I}\mathbf{J}}}, \quad \widetilde{G}^{\mathbf{I}\mathbf{J}, \mathbf{R}\mathbf{S}}=\frac{\partial^2 \big(-\left[\frac{\sigma_{k}(W)}{\sigma_{l}(W)}\right]^{-\frac{1}{k-l}}\big)}{\partial W_{\mathbf{I}\mathbf{J}} \partial W_{\mathbf{R}\mathbf{S}}}.\\
\end{aligned}
\end{eqnarray*}
Assume that $\{a_{ij}\}$ is diagonal. Then $W$ is also diagonal. Combining the above notation with \eqref{in-2} we obtain
\begin{align*}
&\sum_{i,j,r,s=1}^{n}\left(\frac{\partial^{2}\big(-\left[\frac{\sigma_{k}(W)}{\sigma_{l}(W)}\right]^{-\frac{1}{k-l}}\big)}{\partial a_{ij}\partial a_{rs}}
+2\frac{\partial\big(-\left[\frac{\sigma_{k}(W)}{\sigma_{l}(W)}\right]^{-\frac{1}{k-l}}\big)}{\partial a_{ir}}a^{js}\right)\xi_{ij}\xi_{rs} \\
=&\sum_{\mathbf{I},\mathbf{J},\mathbf{R},\mathbf{S}\in \mathfrak{J}}\widetilde{G}^{\mathbf{I}\mathbf{J},\mathbf{R}\mathbf{S}}\left(\sum_{i,j=1}^{n}\frac{\partial W_{\mathbf{I}\mathbf{J}}}{\partial a_{ij}}\xi_{ij}\right)\left(\sum_{r,s=1}^{n}\frac{\partial W_{\mathbf{R}\mathbf{S}}}{\partial a_{rs}}\xi_{rs}\right)+2\sum_{i,j,r=1}^{n}\sum_{\mathbf{P},\mathbf{Q}\in \mathfrak{J}}\widetilde{G}^{\mathbf{P}\mathbf{Q}}\frac{\partial W_{\mathbf{P}\mathbf{Q}}}{\partial a_{ir}}\frac{\xi_{ij}\xi_{rj}}{a_{jj}}\\
\geq&-2\sum_{\mathbf{I},\mathbf{J},\mathbf{R},\mathbf{S}\in \mathfrak{J}}\widetilde{G}^{\mathbf{I}\mathbf{R}}
\left(\sum_{i,j=1}^{n}\frac{\partial W_{\mathbf{I}\mathbf{J}}}{\partial a_{ij}}\xi_{ij}\right)\left(\sum_{r,s=1}^{n}\frac{\partial W_{\mathbf{R}\mathbf{S}}}{\partial a_{rs}}\xi_{rs}\right)W^{\mathbf{J}\mathbf{S}}+2\sum_{i,j,r=1}^{n}\sum_{\mathbf{P}\in \mathfrak{J}}\widetilde{G}^{\mathbf{P}\mathbf{P}}
\frac{\partial W_{\mathbf{P}\mathbf{P}}}{\partial a_{ir}}\frac{\xi_{ij}\xi_{rj}}{a_{jj}}\\
=&-2\sum_{\mathbf{I},\mathbf{J}\in \mathfrak{J}}\widetilde{G}^{\mathbf{I}\mathbf{I}}\frac{\left(\sum_{i,j=1}^{n}\frac{\partial W_{\mathbf{I}\mathbf{J}}}{\partial a_{ij}}\xi_{ij}\right)^2}{W_{\mathbf{J}\mathbf{J}}}
+2\sum_{i,j=1}^{n}\sum_{\mathbf{I}\in \mathfrak{J}}\widetilde{G}^{\mathbf{I}\mathbf{I}}
\frac{\partial W_{\mathbf{I}\mathbf{I}}}{\partial a_{ii}}\frac{\xi_{ij}^{2}}{a_{jj}}\\
=&2\sum_{\mathbf{I}\in \mathfrak{J}} \widetilde{G}^{\mathbf{I}\mathbf{I}}\left\{-\frac{\left({\sum_{i=1}^{n}\frac {\partial W_{\mathbf{I}\mathbf{I}}}{\partial a_{ii}}\xi_{ii}}\right)^2}{{W_{\mathbf{I}\mathbf{I}}}}
-\sum _{|\mathbf{J}\cap \mathbf{I}|=\mathscr{P}-1}\frac{\left({\sum_{i,j=1}^{n}\frac {\partial W_{\mathbf{I}\mathbf{I}}}{\partial a_{ij}}\xi_{ij}}\right)^2}{W_{\mathbf{J}\mathbf{J}}}
+2\sum_{i,j=1}^{n}\frac {\partial W_{\mathbf{I}\mathbf{I}}}{\partial a_{ii}}\frac {\xi_{ij}^2}{a_{jj}}\right\}\\
=&2\sum_{\mathbf{I}\in \mathfrak{J}} \widetilde{G}^{\mathbf{I}\mathbf{I}}\left\{-\frac{\left(\sum_{i\in \mathbf{I}}\xi_{ii}\right)^2}{\sum_{i\in \mathbf{I}}a_{ii}}-\sum_{i\in \mathbf{I}}\sum_{j\in \overline{\mathbf{I}}}\frac{\xi_{ij}^2}{\sum_{s\in \mathbf{I}-i+j}a_{ss}}
+\sum_{i\in \mathbf{I}}\frac{\xi_{ii}^2}{a_{ii}}+\sum_{i\in \mathbf{I}}\bigg(\sum_{j\in {\mathbf{I}-i}}\frac{\xi_{ij}^2}{a_{jj}}+\sum_{j\in {\overline{\mathbf{I}}}}\frac{\xi_{ij}^2}{a_{jj}}
\bigg)\right\}\\
\geq&2\sum_{\mathbf{I}\in \mathfrak{J}}\widetilde{G}^{\mathbf{I}\mathbf{I}}\left\{-\frac{\left(\sum_{i\in \mathbf{I}}\frac{\xi_{ii}}{\sqrt{a_{ii}}}\sqrt{a_{ii}}\right)^2}{\sum_{i\in \mathbf{I}}a_{ii}}+\sum_{i\in \mathbf{I}}\frac{\xi_{ii}^2}{a_{ii}}\right\}\\
\geq&2\sum_{\mathbf{I}\in \mathfrak{J}} \widetilde{G}^{\mathbf{I}\mathbf{I}}\left\{-\frac{\sum_{i\in \mathbf{I}}\frac{\xi_{ii}^2}{{a_{ii}}}\sum_{i\in \mathbf{I}}{a_{ii}}}{\sum_{i\in \mathbf{I}}a_{ii}}+\sum_{i\in \mathbf{I}}\frac{\xi_{ii}^2}{a_{ii}}\right\}\\
\geq&0,
\end{align*}
where we apply Proposition \ref{p1} and Proposition \ref{p2} in step 2 and step 3, respectively, completes the proof.
\end{proof}

%%%%%%%%%%%%%%%%%%%%%%%%%%%%%%%%%%%%%%%%%%%%%%%%%%%%%%%%%%%%%%%%%%%%%%%%%%%%%%%%%%%%%%%%%%%%%%%%%%%%%%%%%%%%%%%%%%%%%%%%%%%%%%%%%%%%%%%%%%%%%%%%%%%%%%%%%%%%%%%%%%%%
\section{Full rank theorem}
%%%%%%%%%%%%%%%%%%%%%%%%%%%%%%%%%%%%%%%%%%%%%%%%%%%%%%%%%%%%%%%%%%%%%%%%%%%%%%%%%%%%%%%%%%%%%%%%%%%%%%%%%%%%%%%%%%%%%%%%%%%%%%%%%%%%%%%%%%%%%%%%%%%%%%%%%%%%%%%%%%%%

This section is devoted to proving the full rank theorem for equation (1.3). The theorem plays a fundamental role in establishing the existence of strictly spherically convex solutions, since a particular form of it will be used to ensure that convexity is preserved throughout the continuity method.
\begin{lemma}\label{Full-rank-1}
Let $u$ be a positive admissible solution of the equation \eqref{1123} such that $\{u_{ij}+u\delta_{ij}\}$ is positive semidefinite on $\mathbb{S}^{n}$. Suppose that $\varphi$ is a smooth positive function satisfying one of the following conditions:
\begin{enumerate}
\item[(\romannumeral1)] if $p\geq1$, $q\leq k+1$, $(\varphi ^{-\frac{1}{k-l+p-1}})_{ii}+\varphi ^{-\frac{1}{k-l+p-1}}\geq0$;

\item[(\romannumeral2)] if $p\geq1$, $k+1< q <2k-l+p$, $(\varphi ^{-\frac{1}{k-l+p-1}})_{ii}+\frac{2k-l+p-1}{k-l+p-1}\varphi ^{-\frac{1}{k-l+p-1}}\geq0$.
\end{enumerate}
Then $\{u_{ij}+u\delta_{ij}\}$ is positive definite on $\mathbb{S}^{n}$.
\end {lemma}
\begin{proof}
Define $A=\{u_{ij}+u\delta_{ij}\}$. Suppose that $A$ attains its minimal rank $m$ at some point $x_{0}\in\mathbb{S}^{n}$, Then
$$\sigma_{m}(A)(x_{0})>0, \quad   \sigma_{m+1}(A)(x_{0})=0.$$

Then there exists an open neighborhood $\mathcal {O}\subset\mathbb{S}^{n}$ of $x_{0}$ and a positive constant $c$ such that  $$\sigma_{m}(A)(x_{0})\geq c>0,\quad in\,\mathcal {O}.$$
Without loss of generality, we assume $m\leq n-1$, otherwise the proof is complete. Let $\lambda=(\lambda _{1},\cdot\cdot\cdot,\lambda _{n} )$ denote the eigenvalues of $A$. Fix a point \(x\in \mathcal O\). Choose a local orthonormal frame around $x$ such that $A$ is diagonal at $x$, with $a_{11}\le a_{22}\le \cdots \le a_{nn}.$ All the computations below are carried out at this fixed point. With this setup, we define the test function
$$\phi (x)=\sigma_{m+1}(A)+\frac{\sigma_{m+2}(A)}{\sigma_{m+1}(A)}.$$
It is proved in section 2 in \cite{BP} that $\phi$ is in $C^{1,1}$.
Adopting the notations of \cite{P}, we write $h(y)\lesssim k(y)$ if there exist positive constants $c_{1}$ and $c_{2}$ such that
$$(h-k)(y)\lesssim(c_{1}|\nabla\phi|+c_{2}\phi)(y),$$
with $h(y)\sim k(y)$ meaning that
$$h(y)\lesssim k(y),\quad
k(y)\lesssim h(y).$$
Define the index sets for the eigenvalues as $B=\{1,\cdot\cdot\cdot,n-m\}$ and $G=\{n-m+1,n-m+2,\cdot\cdot\cdot,n\}$. Their corresponding eigenvalues, $\lambda_{B}=( \lambda_{1},\cdot\cdot\cdot.\lambda_{n-m} )$ and $\lambda_{G}=(\lambda_{n-m+1},\cdot\cdot\cdot,\lambda_{n} )$ are referred to as the bad and good eigenvalues of $A$. When clear from context, we identify $B$ with $\lambda_{B}$  and $G$ with $\lambda_{G}$. This yields
$$0\sim\phi(x)\sim \sigma_{m+1}(A)\sim \sigma_{m}(G)\sigma_{1}(B)\sim \sigma_{1}(B).$$
For the convenience of calculation, equation \eqref{1123} can be rewritten as
\begin{equation}\label{eq3210}
\begin{split}
\widehat{F}:=-[\frac{\sigma_k(\Lambda(\nabla^{2}u+uI))}{\sigma_l(\Lambda(\nabla^{2}u+uI))}]^{-\frac{1}{k-l}}=u^{t}(u+|\nabla u|^{2})^{\frac{s}{2}}\widehat{\varphi}(x):=\widehat{f},
\end{split}
\end{equation}
where $t=-\frac{p-1}{k-l}$,$s=-\frac{k+1-q}{k-l}$ and $\widehat{\varphi}(x)=-\varphi^{-\frac{1}{k-l}}(x)$.
We denote
$$\widehat{F}^{ij}=\frac{\partial \widehat{F}}{\partial a_{ij}},\quad \widehat{F}^{ij,rs}=\frac{\partial^{2} \widehat{F}}{\partial a_{ij}\partial a_{rs}}, \quad \widehat{f}_{i}=\frac{\partial \widehat{f}}{\partial x_{i}},\quad \widehat{f}_{ij}=\frac{\partial^{2} \widehat{f}}{\partial x_{i}\partial x_{j}}.$$
Differentiating equation \eqref{eq3210} twice yields
$$\widehat{F}^{\alpha\beta}a_{\alpha \beta i}=\widehat{f}_{i},\quad \widehat{F}^{\alpha\beta}a_{\alpha \beta ii}+\widehat{F}^{\alpha\beta,rs}a_{\alpha \beta i}a_{rs i}=\widehat{f}_{ii}.$$
It follows from Theorem 3.2 in \cite{BP} that
\begin{equation}\label{eq3211}
\begin{split}
\widehat{ F}^{\alpha\beta}\phi_{\alpha\beta}=
 &\mathcal {O}(\phi+\sum_{i,j\in B}|\nabla a_{ij}|)-\frac{1}{\sigma_{1}(B)}\sum_{\alpha}\sum_{i\neq j\in B}\widehat{F}^{\alpha\alpha}a_{ij\alpha}^{2} \\
 &-\frac{1}{\sigma_{1}(B)^{3}} \sum_{\alpha} \sum_{i\in B} \widehat{F}^{\alpha \alpha} (a_{ii\alpha}{\sigma_{1}(B)}-a_{ii}\sum_{j \in B}a_{jj\alpha})^{2} \\
 &-2\sum_{i\in B}(\sigma_{m}(G)+\frac{\sigma_{1}(B|i)^{2}-\sigma_{2}(B|i)}{\sigma_{1}(B)^{2}})\sum_{\alpha,j\in G}\widehat{F}^{\alpha\alpha}\frac{a_{j\alpha i}^{2}}{a_{jj}} \\
 &+\sum_{i\in B}(\sigma_{m}(G)+\frac{\sigma_{1}(B|i)^{2}-\sigma_{2}(B|i)}{\sigma_{1}(B)^{2}})\sum_{\alpha}\widehat{F}^{\alpha\alpha}a_{ii\alpha\alpha}.
\end{split}
\end{equation}
For any $ i \in B $, the Ricci identity gives
\begin{equation}\label{eq3212}
\begin{split}
\widehat{F}^{\alpha\alpha}a_{ii\alpha\alpha}&=\widehat{F}^{\alpha\alpha}(a_{\alpha\alpha ii}+a_{ii}-a_{\alpha\alpha})\\
&=\widehat{F}^{\alpha\alpha}a_{\alpha\alpha ii}+\widehat{F}^{\alpha\alpha}a_{ii}-\widehat{F}^{\alpha\alpha}a_{\alpha\alpha}\\
&=\widehat{f}_{ii}-\widehat{F}^{\alpha\beta,rs}a_{\alpha\beta i}a_{rsi}+\widehat{f}+\mathcal {O}(\phi)\\
&=-\sum_{\alpha,\beta,r,s\in G}\widehat{F}^{\alpha\beta,rs}a_{\alpha\beta i}a_{rsi}+\widehat{f}_{ii}+\widehat{f}+\mathcal {O}(\phi+\sum_{i,j\in B}|\nabla a_{ij}|).\\
\end{split}
\end{equation}
We claim that
\begin{equation}\label{eq3213}
\begin{split}
\sum_{i\in B}(\widehat{f}_{ii}+\widehat{f})\leq \mathcal {O}(\phi+\sum_{i,j\in B}|\nabla a_{ij}|).
\end{split}
\end{equation}
By direct computation, we obtain
\begin{eqnarray*}
\begin{split}
\sum_{i\in B}(\widehat{f}_{ii}+\widehat{f})=&\sum_{i\in B}(\widehat{f}_{x_{i}x_{i}}+2\widehat{f}_{x_{i} z}u_{i}+\widehat{f}_{zz}u_{i}^{2})+\sum_{k}\sum_{i\in B}\widehat{f}_{p_{k}}u_{iik}+\sum_{i\in B}\widehat{f}\\
&+\sum_{i\in B}u_{ii}(2\widehat{f}_{x_{i}p_{i}}+\widehat{f}_{z}+2\widehat{f}_{zp_{i}}u_{i}+\widehat{f}_{p_{i}p_{i}}u_{ii}).\\
\end{split}
\end{eqnarray*}
The relation $a_{ii}=u_{ii}+u=\mathcal {O}(\phi)$ implies that
\begin{align*}
\sum_{i\in B}(\widehat{f}_{ii}+\widehat{f})=\,&\mathcal {O}(\phi+ \sum_{i,j\in B}| \nabla a_{ij}|)+\sum_{i\in B}(\widehat{f}_{x_{i}x_{i}}+2\widehat{f}_{x_{i} z}u_{i}+\widehat{f}_{zz}u_{i}^{2}\\
&-2\widehat{f}_{x_{i}p_{i}}u-\widehat{f}_{z}u-2\widehat{f}_{zp_{i}}uu_{i}+\widehat{f}_{p_{i}p_{i}}u^{2}-\sum_{k}\widehat{f}_{p_{k}}u_{k}+\widehat{f}).
\end{align*}
Denote
\begin{align*}
H:=\sum_{i\in B}(&\widehat{f}_{x_{i}x_{i}}+2\widehat{f}_{x_{i}z}u_{i}+\widehat{f}_{zz}u_{i}^{2}-2\widehat{f}_{x_{i}p_{i}}u-\widehat{f}_{z}u\\
&-2\widehat{f}_{zp_{i}}uu_{i}+\widehat{f}_{p_{i}p_{i}}u^{2}-\sum_{k}\widehat{f}_{p_{k}}u_{k}+\widehat{f}),
\end{align*}
To establish the claim, it is enough to prove $H\leq0$. After a lengthy but straightforward computation, one obtains
\begin{equation}\label{eq3214}
\begin{split}
H=\,&u^{t}(u^{2}+|\nabla u|^{2})^{\frac{s}{2}}\sum_{i\in B}\bigg(\widehat{\varphi}_{ii}+\frac{2tu_{i}\widehat{\varphi}_{i}}{u}+\frac{t(t-1)u_{i}^{2}\widehat{\varphi}}{u^{2}}\\
&+\frac{su_{i}^{2}\widehat{\varphi}}{u^{2}+|\nabla u|^{2}}-\frac{s\sum_{k}u_{k}^{2}}{u^{2}+|\nabla u|^{2}}+(1-t)\widehat{\varphi}\bigg),
\end{split}
\end{equation}
where $t,\,s,\,\widehat{\varphi}$ are defined in \eqref{eq3210}. Note that
\begin{equation}\label{eq3215}
\begin{split}
\frac{2tu_{i}\widehat{\varphi}_{i}}{u}+\frac{t(t-1)u_{i}^{2}\widehat{\varphi}}{u^{2}}\leq\frac{t\widehat{\varphi}_{i}^{2}}{(1-t)\widehat{\varphi}},
\end{split}
\end{equation}
for $t\leq0$ or $t>1$. Next we continue the proof with two cases.\\
$Case\,1 :t\leq0,\,s\leq0,\, i.e.,\,p\geq1,q\leq k+1.$\\
From $m\leq n-1$, it follows that
\begin{eqnarray}\label{eq3216}
\begin{split}
\sum_{i\in B}\bigg(\frac{su_{i}^{2}\widehat{\varphi}}{u^{2}+|\nabla u|^{2}}-\frac{s\sum_{k}u_{k}^{2}\widehat{\varphi}}{u^{2}+|\nabla u|^{2}}\bigg)&=\frac{s\widehat{\varphi}\sum_{i\in B}u_{i}^{2}}{u^{2}+|\nabla u|^{2}}-\frac{s(n-m)\widehat{\varphi}\sum_{k}u_{k}^{2}}{u^{2}+|\nabla u|^{2}}\\
&\leq\frac{s\widehat{\varphi}\sum_{i\in B}u_{i}^{2}}{u^{2}+|\nabla u|^{2}}-\frac{s\widehat{\varphi}\sum_{k}u_{k}^{2}}{u^{2}+|\nabla u|^{2}}\\
&=-\frac{s\widehat{\varphi}\sum_{i\in G}u_{i}^{2}}{u^{2}+|\nabla u|^{2}}\\
&\leq0.
\end{split}
\end{eqnarray}
Combining with \eqref{eq3214} ,\eqref{eq3215}, \eqref{eq3216} and assumption (1) on $\varphi$ yields
\begin{eqnarray*}
\begin{split}
H\leq& u^{t}(u^{2}+|\nabla u|^{2})^{\frac{s}{2}}\sum_{i\in B}\bigg(\widehat{\varphi}_{ii}+\frac{t\widehat{\varphi}_{i}^{2}}{(1-t)\widehat{\varphi}}+(1-t)\widehat{\varphi}\bigg)\\
=&-\frac{k-l+p-1}{k-l}u^{-\frac{p-1}{k-l}}(u^{2}+|\nabla u|^{2})^{-\frac{k+1-q}{2(k-l)}}\varphi^{-\frac{p-1}{(k-l)(k-l+p-1)}}\\
&\cdot\sum_{i\in B}\bigg((\varphi^{-\frac{1}{k-l+p-1}})_{ii}+(\varphi^{-\frac{1}{k-l+p-1}})\bigg) \\
\leq&0.
\end{split}
\end{eqnarray*}
$Case\,2 :t\leq0,\,s>0,\,i.e.,\,p\geq1,q> k+1.$
\begin{align} \label{eq3217}
\sum_{i\in B}\bigg(\frac{su_{i}^{2}\widehat{\varphi}}{u^{2}+|\nabla u|^{2}}-\frac{s\sum_{k}u_{k}^{2}\widehat{\varphi}}{u^{2}+|\nabla u|^{2}}\bigg)\leq-s\sum_{i\in B}\widehat{\varphi}.
\end{align}
By combining \eqref{eq3214} ,\eqref{eq3215}, \eqref{eq3217} and assumption (2) on $\varphi$, we obtain
\begin{eqnarray*}
\begin{split}
H\leq& u^{t}(u^{2}+|\nabla u|^{2})^{\frac{s}{2}}\sum_{i\in B}\bigg(\widehat{\varphi}_{ii}+\frac{t\widehat{\varphi}_{i}^{2}}{(1-t)\widehat{\varphi}}+(1-t-s)\widehat{\varphi}\bigg)\\
=&-\frac{k-l+p-1}{k-l}u^{-\frac{p-1}{k-l}}(u^{2}+|\nabla u|^{2})^{-\frac{k+1-q}{2(k-l)}}\varphi^{-\frac{p-1}{(k-l)(k-l+p-1)}}\\
&\cdot\sum_{i\in B}\bigg((\varphi^{-\frac{1}{k-l+p-1}})_{ii}+\frac{2k-l+p-q}{k-l+p-1}(\varphi^{-\frac{1}{k-l+p-1}})\bigg) \\
\leq&0.
\end{split}
\end{eqnarray*}
Hence claim \eqref{eq3213} is proved. It follows from \eqref{eq3211}, \eqref{eq3212} and \eqref{eq3213} that
\begin{equation}\label{eq3220}
\begin{split}
\widehat{F}^{\alpha\beta}\phi_{\alpha\beta} \leq &\mathcal {O}(\phi+\sum_{i,j\in B}|\nabla a_{ij}|)-\frac{1}{\sigma_{1}(B)}\sum_{\alpha}\sum_{i\neq j\in B}\widehat{F}^{\alpha\alpha}a_{ij\alpha}^{2}\\
&-\frac{1}{\sigma_{1}(B)^{3}} \sum_{\alpha} \sum_{i\in B} \widehat{F}^{\alpha \alpha} \bigg(a_{ii\alpha}{\sigma_{1}(B)}-a_{ii}\sum_{j \in B}a_{jj\alpha}\bigg)^{2} \\
&-\sum_{i\in B}\bigg(\sigma_{m}(G)+\frac{\sigma_{1}(B|i)^{2}-\sigma_{2}(B|i)}{\sigma_{1}(B)^{2}}\bigg)\\
&\cdot\bigg(2\sum_{\alpha,j\in G}\widehat{F}^{\alpha\alpha}\frac{a_{j\alpha i}^{2}}{a_{jj}}+\sum_{\alpha,\beta,r,s\in G}\widehat{F}^{\alpha\beta,rs}a_{\alpha\beta i}a_{rsi}\bigg).
\end{split}
\end{equation}
By proposition \ref{L1} we know that $\widehat{F}(W^{-1})$ is inverse convex with respect to $\{a_{ij}\}$. It is equivalent to the following inequality
$$\widehat{F}^{\alpha\beta,rs}X_{\alpha\beta}X_{rs}+2\frac{\widehat{F}^{\alpha r}}{a_{\beta s}} X_{\alpha\beta} X_{rs} \geq0,\quad \forall (X_{\alpha\beta})\in Sym(n).$$
Taking
 \[
\begin{cases}
X_{\alpha\beta}=-a_{\alpha\beta i} \quad \alpha,\beta\in G\\
X_{\alpha\beta}=0 \quad \quad  otherwise,
\end{cases}
\]
then we have
\begin{equation}\label{eq3221}
\begin{split}
2\sum_{\alpha,j\in G}\widehat{F}^{\alpha\alpha}\frac{a_{j\alpha i}^{2}}{a_{jj}}+\sum_{\alpha,\beta,r,s\in G}\widehat{F}^{\alpha\beta,rs}a_{\alpha\beta i}a_{rsi}\geq0.
\end{split}
\end{equation}
The combination of \eqref{eq3220} and \eqref{eq3221} leads to the conclusion that
\begin{equation*}
\begin{split}
\widehat{F}^{\alpha\beta}\phi_{\alpha\beta} \leq &\mathcal {O}(\phi+\sum_{i,j\in B}|\nabla a_{ij}|)-\frac{1}{\sigma_{1}(B)}\sum_{\alpha}\sum_{i\neq j\in B}\widehat{F}^{\alpha\alpha}a_{ij\alpha}^{2}\\
&\quad-\frac{1}{\sigma_{1}(B)^{3}} \sum_{\alpha} \sum_{i\in B} \widehat{F}^{\alpha \alpha} \bigg(a_{ii\alpha}{\sigma_{1}(B)}-a_{ii}\sum_{j \in B}a_{jj\alpha}\bigg)^{2} \\
\leq& \mathcal {O}(\phi+\sum_{i,j\in B}|\nabla a_{ij}|)\\
\leq& \mathcal {O}(\phi+|\nabla\phi|).
\end{split}
\end{equation*}
The strong minimum principle gives $\phi\equiv0 $ in $\mathcal {O}$, making $\{x:\phi=0\}$ open and closed, and so $\phi\equiv0$ on $\mathbb{S}^{n}$. Thus, $A=\{u_{ij}+u\delta\mathcal{}_{ij}\}$ is of constant rank. The Minkowski integral formula then forces this rank to be full, proving the theorem.
\end {proof}
%%%%%%%%%%%%%%%%%%%%%%%%%%%%%%%%%%%%%%%%%%%%%%%%%%%%%%%%%%%%%%%%%%%%%%%%%%%%%%%%%%%%%%%%%%%%%%%%%%%%%%%%%%%%%%%%%%%%%%%%%%%%%%%%%%%%%%%%%%%%%%%%%%%%%%%%%%%%%%%%%%%%
\section{the nonhomogeneous case $p>q-l$}
%%%%%%%%%%%%%%%%%%%%%%%%%%%%%%%%%%%%%%%%%%%%%%%%%%%%%%%%%%%%%%%%%%%%%%%%%%%%%%%%%%%%%%%%%%%%%%%%%%%%%%%%%%%%%%%%%%%%%%%%%%%%%%%%%%%%%%%%%%%%%%%%%%%%%%%%%%%%%%%%%%%%
In this section, we derive the a priori estimates and establish the existence and uniqueness of solutions to the equation \eqref{1123} in the nonhomogeneous case $p>q-l$.
\subsection{$C^{0}$ estimate}
\begin {theorem}\label{L5}
Let $0\leq l<k \leq N$, $n\geq2$, $\mathscr{P}>1$, and $p> q-l$. Let $\varphi\in C^{0}(\mathbb{S}^{n})$ be a positive function. Suppose that $u\in C^{2}(\mathbb{S}^{n})$ is a positive admissible solution of the equation \eqref{1123}. Then
$$\frac{C_{N}^{k}}{C_{N}^{l}}\frac{\mathscr{P}^{k-l}}{\max_{\mathbb{S}^{n}}\varphi}\leq u^{p-q+l}(x)
\leq\frac{C_{N}^{k}}{C_{N}^{l}}\frac{\mathscr{P}^{k-l}}{\min_{\mathbb{S}^{n}}\varphi},\quad \forall x\in\mathbb{S}^{n}.$$
\end {theorem}
\begin {proof}
Assume that $\min _{\mathbb{S}^{n}}u(x)$ is attained at $x_{0}$. Then, at $x_{0}$,
$$|\nabla u|=0,\quad \quad \nabla^{2}u\geq0,$$
 which implies that
$\lambda_i(\nabla^{2}u+uI)\geq u$ and $\Lambda_{\mathbf{J}}\geq \mathscr{P}u$
for any $i=1,\cdots, n$ and $\mathbf{J}\in \mathfrak{I}$. It follows from equation \eqref{1123} that
$$u^{p-1}(u^{2}+|\nabla u|^{2})^{\frac{k+1 -q}{2}}\varphi=\frac{\sigma_k(\Lambda)}{\sigma_l(\Lambda)}\geq\frac{C_{N}^{k}}{C_{N}^{l}}\mathscr{P}^{k-l}u^{k-l},$$
using $|\nabla u|=0$ yields
$$u^{p-q+l}(x_{0})\geq \frac{C_{N}^{k}}{C_{N}^{l}}\frac{\mathscr{P}^{k-l}}{\varphi(x_{0})}\geq\frac{C_{N}^{k}}{C_{N}^{l}}\frac{\mathscr{P}^{k-l}}{\max_{\mathbb{S}^{n}}\varphi}.$$
Similarly, suppose $\max_{\mathbb{S}^{n}}u(x)$ is attained at $x_{1}$.  A simple derivation reveals that
$$u^{p-1}(u^{2}+|\nabla u|^{2})^{\frac{k+1-q}{2}}\varphi=\frac{\sigma_k(\Lambda)}{\sigma_l(\Lambda)}\leq\frac{C_{N}^{k}}{C_{N}^{l}}\mathscr{P}^{k-l}u^{k-l},$$
and hence
$$u^{p-q+l}(x_{1})\leq \frac{C_{N}^{k}}{C_{N}^{l}}\frac{\mathscr{P}^{k-l}}{\varphi(x_{1})}\leq\frac{C_{N}^{k}}{C_{N}^{l}}\frac{\mathscr{P}^{k-l}}{\min_{\mathbb{S}^{n}}\varphi}.$$
\end {proof}

\subsection{$C^{1}$ estimate}
\begin {theorem}\label{L9}
Let $0\leq l<k \leq N$, $n\geq2$, $\mathscr{P}>1$, and $p> q-l$. Let $\varphi\in C^{1}(\mathbb{S}^{n})$ be a positive function. Suppose that $u\in C^{3}(\mathbb{S}^{n})$ is a positive admissible solution of the equation \eqref{1123}. Then
$$\frac{|\nabla u(x)|}{u(x)}\leq\frac{1}{p-q+l}\max_{\mathbb{S}^{n}}\frac{|\nabla\varphi|}{\varphi},\quad \forall x\in\mathbb{S}^{n}.$$
\end {theorem}
\begin {proof}
For the positive admissible solution $u$, set $v=\log u$. Then it suffices to estimate $|\nabla v|$, since $|\nabla v|=\frac{|\nabla u|}{u}$. A simple calculation shows that
\begin{equation}\label{32111}
\begin{split}
 a_{ij}:=u_{ij}+u\delta_{ij}=e^{v}(v_{ij}+v_{i}v_{j}+\delta_{ij}).
\end{split}
\end{equation}
Let $\{W_{\mathbf{I}\mathbf{J}}\}_{\mathbf{I}\mathbf{J}\in\mathfrak{J}}$ be the linear derivation of $\{a_{ij}\}$. Denote
\begin{equation*}
\begin{split}
F=\frac{\sigma_{k}(W)}{\sigma_{l}(W)},
\quad F^{ij}=\frac{\partial [\frac{\sigma_{k}(W)}{\sigma_{l}(W)}]}{\partial a_{ij}},\quad F^{\mathbf{I}\mathbf{J}}=\frac{\partial [\frac{\sigma_{k}(W)}{\sigma_{l}(W)}]}{\partial W_{\mathbf{I}\mathbf{J}}}.
\end{split}
\end{equation*}
Consider the test function
$$P=|\nabla v|^{2}.$$
Assume that $P$ attains its maximum at point $x_{0}$. If $P(x_0)=0$, then the desired estimate is trivial. Hence we may assume that $P(x_0)>0$, and choose a local orthonormal frame field such that at $x_{0}$
\begin{equation*}\label{3218}
\begin{split}
v_{1}=|\nabla v|>0,\quad \{v_{ij}\}_{2\leq i,j\leq n} \,\,\mbox{ is diagonal}.
\end{split}
\end{equation*}
In what follows, all the calculations are performed at $x_{0}$. Hence
\begin{equation*}\label{32114}
\begin{split}
0=P_{i}=\nabla_{i}(|\nabla v|^{2})=\sum_{k=1}^{n}2v_{k}v_{ki}=2v_{1}v_{1i},\\
\end{split}
\end{equation*}
and
 \begin{equation*}
\begin{split}
0\geq P_{ii}=2\sum_{k=1}^{n}(v_{k}v_{kii}+v_{ki}^{2}).
\end{split}
\end{equation*}
Given $v_{1}>0$, it follows that $v_{1i}=0$ for $i=1,2,\cdot\cdot\cdot,n.$ Hence $\{v_{ij}\}_{1\leq i,j\leq n}$ is diagonal, and the same is true for $\{a_{ij}\}_{1\leq i,j\leq n}$ and $\{F^{ij}\}_{1\leq i,j\leq n}$. Thus
\begin{equation*}\label{3219}
\begin{split}
0\geq F^{ii}P_{ii}=2v_{1}F^{ii}v_{1ii}+2F^{ii}v_{ii}^{2}\geq2v_{1}F^{ii}v_{1ii},
\end{split}
\end{equation*}
which yields that $F^{ii}v_{1ii}\leq0$. By the Ricci identity,
\begin{equation*}\label{32110}
\begin{split}
v_{1ii}=v_{ii1}+v_{1}-v_{i}\delta_{1i}.
\end{split}
\end{equation*}
Combining this with the above inequality, we obtain
\begin{equation}\label{32112}
\begin{split}
0\geq F^{ii}v_{1ii}=F^{ii}(v_{ii1}+v_{1}-v_{i}\delta_{1i})=F^{ii}v_{ii1}+v_{1}\sum_{i=2}^{n}F^{ii}\geq F^{ii}v_{ii1}.
\end{split}
\end{equation}
On the other hand, using \eqref{32111} and $v_{1i}=0$, we have
\begin{equation}\label{32152}
\begin{split}
a_{ii1}=\nabla_{1}(e^{v}(v_{ii}+v_{i}^{2}+1))=v_{1}a_{ii}+e^{v}v_{ii1}.
\end{split}
\end{equation}
Denote $\Phi=e^{2v}+e^{2v}|\nabla v|^{2}$. Using Proposition \ref{p2} and equation \eqref{1123} yields
\begin{equation}\label{32125}
\begin{split}
F^{ii}a_{ii}=&\sum_{i}\sum_{\mathbf{I}}\frac{\partial F}{\partial W_{\mathbf{I}\mathbf{I}}}\frac{\partial W_{\mathbf{I}\mathbf{I}}}{\partial a_{ii}}a_{ii}
=\sum_{i}\sum_{\{\mathbf{I}\in\mathfrak{J}|i\in \mathbf{I}\}}F^{\mathbf{I}\mathbf{I}}a_{ii}=\sum_{\mathbf{I}}F^{\mathbf{I}\mathbf{I}}W_{\mathbf{I}\mathbf{I}}\\
=&(k-l)F=(k-l)\frac{\sigma_{k}(\Lambda)}{\sigma_{l}(\Lambda)}=(k-l)e^{(p-1)v}\Phi^{\frac{k+1-q}{2}}\varphi,
\end{split}
\end{equation}
and
\begin{equation}\label{32126}
\begin{split}
F^{ii}a_{iij}=\sum_{i}\sum_{\mathbf{I}}\frac{\partial F}{\partial W_{\mathbf{I}\mathbf{I}}}\frac{\partial W_{\mathbf{I}\mathbf{I}}}{\partial a_{ii}}a_{iij}
=\sum_{i}\sum_{\{\mathbf{I}\in\mathfrak{J}|i\in \mathbf{I}\}}F^{\mathbf{I}\mathbf{I}}a_{iij}=\sum_{\mathbf{I}}F^{\mathbf{I}\mathbf{I}}W_{\mathbf{I}\mathbf{I}j}=\nabla_{j}F.
\end{split}
\end{equation}
Taking $j=1$ in equation \eqref{32126} becomes
\begin{equation}\label{32113}
\begin{split}
F^{ii}a_{ii1}=\nabla_{1}(e^{(p-1)v}\Phi^{\frac{k+1-q}{2}}\varphi)
=e^{(p-1)v}\Phi^{\frac{k+1-q}{2}}(\varphi_{1}+(p+k-q)v_{1}\varphi).
\end{split}
\end{equation}
Taken together \eqref{32112}-\eqref{32113} implies
\begin{equation*}
\begin{split}
0\geq F^{ii}v_{ii1}=e^{-v}(F^{ii}a_{ii1}-v_{1}F^{ii}a_{ii})
=e^{(p-2)v}\Phi^{\frac{k+1-q}{2}}(\varphi_{1}+(p-q+l)v_{1}\varphi),
\end{split}
\end{equation*}
this simplifies to
$$0\geq \varphi_{1}+(p-q+l)v_{1}\varphi,$$
which yields that
\begin{equation*}
\begin{split}
\frac{|\nabla u(x)|}{u(x)}\leq v_{1}(x_{0})\leq\frac{1}{p-q+l}\max_{\mathbb{S}^{n}}\frac{|\nabla\varphi|}{\varphi}.
\end{split}
\end{equation*}
\end {proof}

\subsection{$C^{2}$ estimate}
\begin {theorem}\label{L4}
Let $0\leq l<k \leq C_{n-1}^{\mathscr{P}-1}$, $n\geq2$, $\mathscr{P}>1$, and $p>q-l$. Let $\varphi\in C^{2}(\mathbb{S}^{n})$ be a positive  function. Suppose that $u\in C^{4}(\mathbb{S}^{n})$ is a positive admissible solution of the equation \eqref{1123}. Then there exists a positive constant $C$, depending only on $n$, $k$, $l$, $p$, $q$, $\min_{\mathbb{S}^{n}}\varphi$, $\|u\|_{C^{1}}$ and $\|\varphi\|_{C^{2}}$ such that $$\max_{\mathbb{S}^{n}}|\nabla^{2}u|\leq C.$$
\end {theorem}
\begin {proof}
Consider the auxiliary function
$$Q=\log\lambda_{1}+A|\nabla u|^{2}-\log u,$$
where $\lambda_{1}$ is the maximal eigenvalue of $\{u_{ij}+u\delta_{ij}\}$. Assume that $Q$ attains its maximum at some point $x_{0}\in\mathbb{S}^{n}$. Choose a local orthonormal frame  $\{e_{1},e_{2},\cdot\cdot\cdot,e_{n}\}$ near $x_{0}$ such that
\begin{equation*}
\begin{split}
\lambda_{i}=a_{ii}, \quad \lambda_{1}=\cdot\cdot\cdot=\lambda_{m}>\lambda_{m+1}\geq\cdot\cdot\cdot\geq\lambda_{n},\quad \mbox{at }x_{0}.
\end{split}
\end{equation*}
By Lemma 5 in \cite{GCF1999}, suppose that \(\psi\) is a smooth function such that \(\lambda_1 \leq \psi\) everywhere and \(\lambda(x_0) = \psi(x_0)\). Then, at $x_0$,
$$\psi_i = a_{11i}, $$
 and
$$\psi_{ii} \geq a_{11ii} + 2 \sum_{p>m} \frac{a_{1pi}^2}{\lambda_1 - \lambda_p}.$$
In other words, we have
\begin{equation*}
\lambda_{1i} = a_{11i},
\end{equation*}
and
\begin{equation*}
\lambda_{1ii} \geq a_{11ii} + 2 \sum_{p>m} \frac{a_{1pi}^2}{\lambda_1 - \lambda_p}\geq a_{11ii}.
\end{equation*}
in the viscosity sense. Without loss of generality, the following computations are performed in the viscosity sense.
Denote
\begin{equation*}
\begin{split}
\widetilde{F}= [\frac{\sigma_{k}(\Lambda(\nabla^{2}u+uI))}{\sigma_{l}(\Lambda(\nabla^{2}u+uI))}]^{\frac{1}{k-l}},\quad
\widetilde{F}^{ij}=\frac{\partial \widetilde{F} }{\partial a_{ij}}.
\end{split}
\end{equation*}
Then the equation \eqref{1123} can be written as
\begin{equation}\label{32139}
\begin{split}
\widetilde{F}= [\frac{\sigma_{k}(\Lambda(\nabla^{2}u+uI))}{\sigma_{l}(\Lambda(\nabla^{2}u+uI))}]^{\frac{1}{k-l}}=\widetilde{f}(x,u,\nabla u),
\end{split}
\end{equation}
where
\begin{equation*}
\begin{split}
\widetilde{f}(x,u,\nabla u)=[u^{p-1}(u^{2}+|\nabla u|^{2})^{\frac{k+1-q}{2}}\varphi(x)]^{\frac{1}{k-l}}.
\end{split}
\end{equation*}

At $x_{0}$, the matrix $\{\widetilde{F}^{ij}\}$, $\{a_{ij}\}$ and $\{u_{ij}\}$ are all diagonal. All subsequent calculations are carried out at $x_{0}$. Hence
\begin{equation}\label{32130}
\begin{split}
0=Q_{i}=\frac{a_{11i}}{a_{11}}+2Au_{i}u_{ii}-\frac{u_{i}}{u},
\end{split}
\end{equation}
and
\begin{equation}\label{32127}
\begin{split}
0\geq \widetilde{F}^{ii}Q_{ii}
\geq&\frac{\widetilde{F}^{ii}a_{11ii}}{a_{11}}-\frac{\widetilde{F}^{ii}a_{11i}^{2}}{a_{11}^{2}}
+2A\widetilde{F}^{ii}u_{ii}^{2}+2A\widetilde{F}^{ii}\sum_{l}u_{l}u_{lii}-\frac{\widetilde{F}^{ii}u_{ii}}{u}
+\frac{\widetilde{F}^{ii}u_{i}^{2}}{u^{2}}.
\end{split}
\end{equation}
First, we estimate the term $\frac{\widetilde{F}^{ii}a_{11ii}}{a_{11}}$ in \eqref{32127}.
 By the Ricci identity
\begin{equation}\label{32151}
\begin{split}
a_{11ii}=a_{ii11}+a_{11}-a_{ii}\geq a_{ii11}.
\end{split}
\end{equation}
Differentiating \eqref{32139} twice at $x_{0}$, and using the concavity of $\widetilde{F}$, we obtain
\begin{equation}\label{32140}
\begin{split}
\widetilde{F}^{ii}a_{ii11}=\widetilde{f}_{11}-\widetilde{F}^{ij,rs}a_{ij1}a_{rs1}\geq \widetilde{f}_{11}.
\end{split}
\end{equation}
Then using \eqref{32126} differentiating the equation once yields
\begin{equation}\label{32135}
\begin{split}
\widetilde{F}^{ii}u_{lii}=\widetilde{f}_{x}+\widetilde{f}_{u}u_{l}+\widetilde{f}_{u_{i}}u_{il}-u_{l}\sum_{i}\widetilde{F}^{ii},
\end{split}
\end{equation}
A direct but lengthy computation then gives
\begin{equation}\label{32129}
\begin{split}
\widetilde{f}_{11} \geq -Mu_{11}^{2}-C+\widetilde{f}_{u_{i}}u_{i11},
\end{split}
\end{equation}
where the constants $M$, $N$ and $C$ depends on $n$, $k$, $l$, $p$, $q$, $\sup u$, $\sup \varphi$, $\min_{\mathbb{S}^{n}}\varphi$,  $\|u\|_{C^{1}}$ and $\|\varphi\|_{C^{2}}$.
Without loss of generality, we may assume that $u_{11}>1$. Combining with \eqref{32151} and \eqref{32140}-\eqref{32129} show that
\begin{equation}\label{32128}
\begin{split}
\frac{\widetilde{F}^{ii}a_{11ii}}{a_{11}}
\geq\frac{\widetilde{F}^{ii}a_{ii11}}{a_{11}}\geq\frac{\widetilde{f}_{11}}{a_{11}}
\geq \frac{-Mu_{11}^{2}-C}{a_{11}}+\frac{\widetilde{f}_{u_{i}}u_{i11}}{a_{11}}
\geq \frac{\widetilde{f}_{u_{i}}u_{i11}}{a_{11}}-Mu_{11}-C.
\end{split}
\end{equation}
Together \eqref{32130}, \eqref{32135}, \eqref{32128} and proposition \ref{P4}  yields
\begin{equation}\label{32137}
\begin{split}
&\frac{\widetilde{f}_{u_{i}}u_{i11}}{a_{11}}+2A\widetilde{F}^{ii}\sum_{l}u_{l}u_{lii}\\
=&\frac{\widetilde{f}_{u_{i}}u_{i}}{u}-\frac{\widetilde{f}_{u_{i}}u_{i}}{a_{11}}+2A\widetilde{f}_{x}\sum_{l}u_{l}
+2A|\nabla u|^{2}\widetilde{f}_{u}-2A|\nabla u|^{2}\sum_{i}\widetilde{F}^{ii}\\
\geq& -AC\sum_{i}\widetilde{F}^{ii}.
\end{split}
\end{equation}
Combining \eqref{32130} with the symmetry $u_{11i}=u_{i11}$ gives
\begin{equation*}
\begin{split}
\frac{a_{11i}^{2}}{a_{11}^{2}}=(\frac{u_{i}}{u}-2Au_{i}u_{ii})^{2}\leq \frac{2u_{i}^{2}}{u^{2}}+8A^{2}u_{i}^{2}u_{ii}^{2},
\end{split}
\end{equation*}
which implies
\begin{equation}\label{32136}
\begin{split}
-\frac{\widetilde{F}^{ii}a_{11i}^{2}}{a_{11}^{2}}+\frac{\widetilde{F}^{ii}u_{i}^{2}}{u^{2}}\geq-\frac{\widetilde{F}^{ii}u_{i}^{2}}
{u^{2}}-8A^{2}\widetilde{F}^{ii}u_{i}^{2}u_{ii}^{2}\geq -C\sum_{i}\widetilde{F}^{ii}-A^{2}C\widetilde{F}^{ii}u_{ii}^{2}.
\end{split}
\end{equation}
Using \eqref{32125} one can get
\begin{equation}\label{32138}
\begin{split}
-\frac{\widetilde{F}^{ii}u_{ii}}{u}=-\frac{\widetilde{F}^{ii}a_{ii}}{u}+\sum_{i}\widetilde{F}^{ii}\geq -C.
\end{split}
\end{equation}
Then substituting \eqref{32128}, \eqref{32137}, \eqref{32136} and \eqref{32138} into \eqref{32127}, we arrive at
\begin{equation*}
\begin{split}
0\geq(A-A^{2}C)Cu_{11}^{2}+A(C_{1}u_{11}^{2}-C_{2})C-Mu_{11}-C,\\
\end{split}
\end{equation*}
choosing $0<A<\frac{1}{C}$ yields
\begin{equation*}
\begin{split}
0\geq AC_{1}u_{11}^{2}-Mu_{11}-AC,\\
\end{split}
\end{equation*}
it follows that $u_{11}\leq C$, which completes the proof.
\end {proof}

Then we cosider the $C^2$ estimates  of the equation \eqref{32150}  for $0\leq l<k\leq N$.

\begin {theorem}\label{L10}
Let $0\leq l<k\leq N$, $n\geq2$, $\mathscr{P}>1$, $p>q-l$, and $q=k+1$. Let $\varphi\in C^{2}(\mathbb{S}^{n})$ be a positive function. Suppose that $u\in C^{4}(\mathbb{S}^{n})$ is a positive admissible solution of the equation \eqref{32150}. Then there exists a positive constant $C$ depending only on $n$, $k$, $l$, $p$, $\min_{\mathbb{S}^{n}}\varphi$, $\|u\|_{C^{1}}$ and $\|\varphi\|_{C^{2}}$ such that
\begin{equation*}
\begin{split}
|\nabla^{2}u(x)|\leq C,\quad  x\in \mathbb{S}^{n}.
\end{split}
\end{equation*}
\end {theorem}
\begin {proof}
Consider the auxiliary function
\begin{equation*}
\begin{split}
R(x)=\lambda_{1},
\end{split}
\end{equation*}
where $\lambda_{1}$ is the maximum eigenvalue of $\{u_{ij}+u\delta_{ij}\}$.
Assume that $R(x)$ attains its maximum at $x_{0}\in\mathbb{S}^{n}$. Choose a local orthonormal frame  $\{e_{1},e_{2},\cdot\cdot\cdot,e_{n}\}$ near $x_{0}$ such that
\begin{equation*}
\begin{split}
\lambda_1=u_{11}+u,\quad \{a_{ij}\}_{1\leq i,j\leq n}\,\,\mbox{ is diagonal}.
\end{split}
\end{equation*}
Then at $x_{0}$ we have
$$0=\widetilde{R}_{i}=a_{11i}=u_{11i}+u_{i},$$
 and
$$0\geq\widetilde{R}_{ii}=a_{11ii},$$ in the viscosity sense.
It follows from \eqref{32139} that for $q=k+1$, $\widetilde{f}=[u^{p-1}\varphi(x)]^{\frac{1}{k-l}}$. By the Ricci identity $a_{11ii}=a_{ii11}+a_{11}-a_{ii}$ together with \eqref{32151} yields
\begin{equation*}
\begin{split}
0\geq\widetilde{F}^{ii}\widetilde{R}_{ii}\geq\widetilde{F}^{ii}a_{ii11}\geq \widetilde{f}_{11}=\nabla_{11}(u^{p-1}\varphi)^{\frac{1}{k-l}}\geq\frac{p-1}{k-l}u^{\frac{p-1}{k-l}-1}\varphi^{\frac{1}{k-l}}u_{11}-C,
\end{split}
\end{equation*}
which implies $u_{11}\leq C$.
\end {proof}

\subsection{\textbf{Existence and uniqueness of the positive solutions}}
\quad
\begin {lemma}\label{32148}
Let $0\leq l<k\leq N$, $n\geq2$, $\mathscr{P}>1$ and $p>q-l$. Let $\varphi\in C^{0}(\mathbb{S}^{n})$ be a positive function. Then the positive admissible solution of the equation \eqref{1123} is unique.
\end {lemma}
\begin {proof}
Suppose $u_{1}$ and $u_{2}$ are two positive admissible solutions of equation \eqref{1123}. Consider the function $G=\frac{u_{1}}{u_{2}}$ and assume $G$ attains its maximum at a point $x_{0}\in \mathbb{S}^{n}$. Then at $x_{0}$
\begin{equation*}
\begin{split}
0=\nabla \log G=\frac{\nabla u_{1}}{u_{1}}-\frac{\nabla u_{2}}{u_{2}},
\end{split}
\end{equation*}
and
\begin{equation*}
\begin{split}
0&\geq\nabla^{2}\log G=\frac{\nabla^{2} u_{1}}{u_{1}}-\frac{\nabla^{2} u_{2}}{u_{2}},
\end{split}
\end{equation*}
which implies  $u_{1}^{-1}(\nabla^{2}u_{1}+u_{1}I)\leq u_{2}^{-1}(\nabla^{2}u_{2}+u_{2}I)$. By the monotonicity of $\sigma_k/\sigma_l$ on the admissible cone, we obtain
\begin{equation*}
\begin{split}
\frac{\sigma_{k}(\Lambda(u_{1}^{-1}(\nabla^{2}u_{1}+u_{1}I)))}{\sigma_{l}(\Lambda(u_{1}^{-1}(\nabla^{2}u_{1}+u_{1}I)))}
\leq
\frac{\sigma_{k}(\Lambda(u_{2}^{-1}(\nabla^{2}u_{2}+u_{2}I)))}{\sigma_{l}(\Lambda(u_{2}^{-1}(\nabla^{2}u_{2}+u_{2}I)))}.
\end{split}
\end{equation*}
Therefore
\begin{equation*}
\begin{split}
1=\frac{u_{1}^{1-p}(u_{1}^{2}+|\nabla u_{1}|^{2})^{\frac{q-1-k}{2}}\frac{\sigma_{k}(\Lambda(\nabla^{2}u_{1}+u_{1}I))}{\sigma_{l}(\Lambda(\nabla^{2}u_{1}+u_{1}I))}}
{u_{2}^{1-p}(u_{2}^{2}+|\nabla u_{2}|^{2})^{\frac{q-1-k}{2}}\frac{\sigma_{k}(\Lambda(\nabla^{2}u_{2}+u_{2}I))}{\sigma_{l}(\Lambda(\nabla^{2}u_{2}+u_{2}I))}}
\leq G^{q-l-p}(x_{0}).
\end{split}
\end{equation*}
The condition $q-l-p<0$ yields $G(x_{0})=\max_{\mathbb{S}^{n}}G \leq 1$. A similar argument shows $\min _{\mathbb{S}^{n}}G\geq1$, hence, $u_{1}\equiv u_{2}$.
\end {proof}

\begin {lemma}\label{32149}
Let $0\leq l<k\leq N$, $n\geq2$, $\mathscr{P}>1$, and $p>q-l$.  Then, at every positive admissible solution $u$ of equation \eqref{1123}, the linearized operator $L_u$ is invertible.
\end {lemma}
\begin {proof}
Suppose that $u$ is a positive admissible solution of equation \eqref{1123}. Denote $a_{ij}=u_{ij}+u\delta_{ij}$ for $1\leq i,j\leq n$.Then the linearized operator of equation \eqref{1123} is
\begin{equation*}
\begin{split}
L_{u}(v)=&\widetilde{F}^{ij}(v_{ij}+v\delta_{ij})-\frac{p-1}{k-l}u^{\frac{p-1}{k-l}-1}(u^{2}+|\nabla u|^{2})^{\frac{k+1-q}{2(k-l)}}\varphi^{\frac{1}{k-l}}v\\
&-\frac{k+1-q}{2(k-l)}u^{\frac{p-1}{k-l}}(u^{2}+|\nabla u|^{2})^{\frac{k+1-q}{2(k-l)}-1}\varphi^{\frac{1}{k-l}}(2uv+2\sum_{l}u_{l}v_{l}).
\end{split}
\end{equation*}
Setting $v=u\omega$ yields
\begin{equation*}
\begin{split}
L_{u}(v)=&\widetilde{F}^{ij}[(u_{ij}+u\delta_{ij})\omega+2u_{i}\omega_{j}+u\omega_{ij}]\\
&-\frac{p-1}{k-l}u^{\frac{p-1}{k-l}}(u^{2}+|\nabla u|^{2})^{\frac{k+1-q}{2(k-l)}}\varphi^{\frac{1}{k-l}}\omega\\
&-\frac{k+1-q}{k-l}u^{\frac{p-1}{k-l}}(u^{2}+|\nabla u|^{2})^{\frac{k+1-q}{2(k-l)}-1}\varphi^{\frac{1}{k-l}}\omega\\
&-\frac{k+1-q}{k-l}u^{\frac{p-1}{k-l}+1}(u^{2}+|\nabla u|^{2})^{\frac{k+1-q}{2(k-l)}-1}\varphi^{\frac{1}{k-l}}\sum_{l}u_{l}v_{l}.\\
\end{split}
\end{equation*}

We now show that $KerL_{u}=\{0\}$. It suffices to show that if $L_{u}(v)=0$, then $\omega=0$.
Suppose $\omega$ attains its maximum at a point $x_{0}$. Then $\nabla\omega(x_{0})=0$ and the Hessian matrix satisfies $\nabla^{2}\omega(x_{0})\leq0$. Consequently, at $x_{0}$
\begin{equation*}
\begin{split}
0=L_{u}(v)=(1-\frac{p-1}{k-l}-\frac{k+1-q}{k-l})\widetilde{F}\omega
+u\widetilde{F}^{ij}\omega_{ij}
\leq\frac{q-l-p}{k-l}\widetilde{F}\omega.
\end{split}
\end{equation*}
Since $p>q-l$, we get $\omega(x_{0})\leq0$. Applying the same argument at a minimum point $x_{1}$ of $\omega$, we obtain $\min\omega(x_{1})\geq0$. Hence $\omega\equiv 0$ and therefore $KerL_{u}=\{0\}$.Since $L_{u}$ is a uniformly elliptic linear operator of index zero, it follows that $L_{u}$ is invertible.
\end {proof}
\begin {proof}[\textbf{{Proof of Theorem \ref{L6}}.}]
As in \cite{HMS-04}, we consider the following equation
\begin{equation}\label{32147}
\begin{split}
\frac{\sigma_{k}(\Lambda(\nabla^{2}u+uI))}{\sigma_{l}(\Lambda(\nabla^{2}u+uI))}=u^{p-1}(u^{2}+|\nabla u|^{2})^{\frac{k+1-q}{2}}\varphi_{t},
\end{split}
\end{equation}
where
\begin{equation*}
\begin{split}
\varphi_{t}=\left((1-t)\left[\frac{C_{N}^{k}}{C_{N}^{l}}\mathscr{P}^{k-l}\right]^{-\frac{1}{k-l+p-1}}+t\varphi^{-\frac{1}{k-l+p-1}}\right)^{-(k-l+p-1)}.
\end{split}
\end{equation*}

We define the set $S$ as follows:
$$S = \left\{ t \in [0, 1] \ |\ \text{equation}\, \eqref{32147}\text{ has a positive strictly spherically convex solution } u_t \right\}.$$

Clearly, $S$ is nonempty, since $u_{0}\equiv1$ is a positive strictly spherically convex solution of equation \eqref{32147} for $t=0$. It remains to show that $S$ is both open and closed.

To prove that $S$ is open, let $t_{0}\in S$, and let $u_{t_{0}}$ be the corresponding positive strictly spherically convex solution. By Lemma \ref{32149}, we have $KerL_{u}=\{0\}$. Therefore, by the implicit function theorem, there exists a neighborhood $\mathscr{N}$ of $t_{0}$ such that, for every $t\in\mathscr{N}$, equation \eqref{32147} admits a positive strictly spherically convex solution $u_{t}$. Hence $\mathscr{N}\in S$, and thus $S$ is open.

Next, we show that $S$ is closed. Let $\{t_i\}_{i=1}^\infty\subset S$ be a sequence such that $t_i\to t_0$, and let $u_{t_i}$ be a positive strictly spherically convex solution of equation \eqref{32147} corresponding to $ t = t_i $. By Theorems \ref{L5}, \ref{L9}, and \ref{L4}, together with the Evans--Krylov theorem and Schauder estimates, we obtain higher order estimates. Consequently, after passing to a subsequence if necessary, $u_{t_i}$ converges to a function $u$, where $u$ is a positive admissible solution of equation \eqref{32147} for $t=t_0$. Clearly, $(u_{ij}+u\delta_{ij})$ is positive semidefinite. Moreover, it is easy to verify that $p$, $q$, and $\varphi_{t_0}$ satisfy the assumptions of Theorem \ref{Full-rank-1}, since $p$, $q$, and $\varphi$ satisfy the assumptions of Theorem \ref{L6}. Therefore, by the full rank theorem (Theorem \ref{Full-rank-1}), $(u_{ij}+u\delta_{ij})$ is in fact positive definite. This proves that $t_0\in S$, and hence $S$ is closed.

Therefore, $S = [0,1]$. In particular, equation \eqref{32147} with $t = 1 $, namely equation \eqref{1123}, admits a positive strictly spherically convex solution. The uniqueness follows directly from Lemma \ref{32148}.
\end {proof}
\begin {proof}[\textbf{{Proof of Theorem \ref{L8}} (\romannumeral1).}]
Based on Theorem \ref{L5}, \ref{L9} and \ref{L10}, higher order estimates can be established by Evans-Krylov and Schauder theory. Then  the proof of existence and uniqueness result is quite similar to that given for Theorem \ref{L6}  and so is omitted.
\end {proof}

%%%%%%%%%%%%%%%%%%%%%%%%%%%%%%%%%%%%%%%%%%%%%%%%%%%%%%%%%%%%%%%%%%%%%%%%%%%%%%%%%%%%%%%%%%%%%%%%%%%%%%%%%%%%%%%%%%%%%%%%%%%%%%%%%%%%%%%%%%%%%%%%%%%%%%%%%%%%%%%%%%%%
\section{the homogeneous case $p=q-l$}
%%%%%%%%%%%%%%%%%%%%%%%%%%%%%%%%%%%%%%%%%%%%%%%%%%%%%%%%%%%%%%%%%%%%%%%%%%%%%%%%%%%%%%%%%%%%%%%%%%%%%%%%%%%%%%%%%%%%%%%%%%%%%%%%%%%%%%%%%%%%%%%%%%%%%%%%%%%%%%%%%%%%
In this section, we consider equation \eqref{1123} in the homogeneous case $p=q-l$. By suitably modifying the method in \cite{GL-99}, we are led to consider the following equation:
\begin{equation}\label{32142}
\begin{split}
\frac{\sigma_{k}(\Lambda(\nabla^{2}u+uI))}{\sigma_{l}(\Lambda(\nabla^{2}u+uI))}=u^{p-1+\varepsilon}(u^{2}+|\nabla u|^{2})^{\frac{k+1-q}{2}}\varphi(x),\quad \forall x\in \mathbb{S}^{n},
\end{split}
\end{equation}
for any small $\varepsilon>0$.
\subsection{$C^{0}$ estimate}
\begin {theorem}\label{L2}
Let $0\leq l<k\leq N$, $n\geq2$, $\mathscr{P}>1$, and $p=q-l>1$. Let $\varphi\in C^{0}(S^{n})$ be a positive function. Suppose that $u\in C^{2}(\mathbb{S}^{n})$ is a positive admissible solution of the equation \eqref{32142}. Then
$$\frac{C_{N}^{k}}{C_{N}^{l}}\frac{\mathscr{P}^{k-l}}{\max_{\mathbb{S}^{n}}\varphi}\leq u^{\varepsilon}(x)
\leq\frac{C_{N}^{k}}{C_{N}^{l}}\frac{\mathscr{P}^{k-l}}{\min_{\mathbb{S}^{n}}\varphi},\quad \forall x\in \mathbb{S}^{n}.$$
\end {theorem}
\begin {proof}
The proof is similar to that of Theorem \ref{L5}.
\end {proof}
\subsection{$C^{1}$ estimate}
\begin {theorem}\label{L3}
Let $0\leq l<k\leq N$, $n\geq 2$, $\mathscr{P}>1$, and $p=q-l>1$. Let $\varphi\in C^{1}(S^{n})$ be a positive function. Suppose that $u\in C^{3}(S^{n})$ is a positive admissible solution of equation \eqref{32142}. Then there exists a positive constant $C$, depending only on $n$, $k$, $l$, $\mathscr{P}$, $\min_{S^{n}}\varphi$, and $\|\varphi\|_{C^{1}(S^{n})}$, but independent of $\varepsilon$, such that
\begin{equation*}
\begin{split}
\frac{|\nabla u(x)|}{u(x)}\leq C,\quad \forall x\in \mathbb{S}^{n}.
\end{split}
\end{equation*}
\end {theorem}
\begin {proof}
For the positive admissible solution $u$, let $v=\log u$. Consider the test function
\begin{equation*}
\begin{split}
P=|\nabla v|^{2}.
\end{split}
\end{equation*}
Assume $P$ attains its maximum at point $x_{0}$. If $P(x_0)=0$, then there is nothing to prove. Hence we may assume that $P(x_0)>0$, and choose a local orthonormal frame field such that, at $x_0$,
\begin{equation*}
\begin{split}
v_{1}=|\nabla v|>0,\quad \{v_{ij}\}_{2\leq i,j\leq n}\quad \mbox{is diagonal}.
\end{split}
\end{equation*}
In what follows, all calculations are carried out at $x_{0}$. Hence
\begin{equation*}
\begin{split}
0=P_{i}=2\sum_{k=1}^{n}v_{k}v_{ki}=2v_{1}v_{1i},
\end{split}
\end{equation*}
and
\begin{equation*}
\begin{split}
0\geq P_{ii}=2\sum_{k=1}^{n}(v_{k}v_{kii}+v_{ki}^{2}).
\end{split}
\end{equation*}
Given $v_{1}>0$, it follows that $v_{1i}=0$ for $i=1,2,\cdot\cdot\cdot,n.$ Hence $\{v_{ij}\}_{1\leq i,j\leq n}$ is diagonal, and the same is true for $\{a_{ij}\}_{1\leq i,j\leq n}$ and $\{F^{ij}\}_{1\leq i,j\leq n}$. Thus
\begin{equation*}\label{32143}
\begin{split}
0\geq F^{ii}P_{ii}=2F^{ii}v_{ii}^{2}+2v_{1}F^{ii}v_{1ii}\geq2v_{1}F^{ii}v_{1ii},
\end{split}
\end{equation*}
applying the Ricci identity $v_{1ii}=v_{ii1}+v_{1}-v_{i}\delta_{1i}$ we obtain
\begin{equation}\label{32153}
\begin{split}
0\geq F^{ii}v_{ii1}+v_{1}\sum_{i=2}^{n}F^{ii}.
\end{split}
\end{equation}
Using proposition \ref{P4} and proposition \ref{p2}, we obtain
\begin{equation}\label{32154}
\begin{split}
\sum_{i=2}^{n}F^{ii}&=\sum_{i=2}^{n}\sum_{\{\mathbf{I}\in\mathfrak{I}|i\in I\}}F^{\mathbf{I}\mathbf{I}}\frac{\partial W_{\mathbf{I}\mathbf{I}}}{\partial a_{ii}}\\
&=(\mathscr{P}-1)\sum_{\mathbf{I}}F^{\mathbf{I}\mathbf{I}}+\sum_{1\notin I}F^{\mathbf{I}\mathbf{I}}\geq\frac{\mathscr{P}-1}{\mathscr{P}}\sum_{i=1}^{n}F^{ii}\\
&\geq C(n, k ,l,\mathscr{P})[e^{(p+\varepsilon+k-q)v}(1+|\nabla v|^{2})^{\frac{k+1-q}{2}})\varphi]^{\frac{k-l-1}{k-l}},
\end{split}
\end{equation}
where $C(n, k ,l,\mathscr{P})$ is a positive constant depending only on $n, k, l$ and $\mathscr{P}$. Arguing as in the derivation of \eqref{32152}-\eqref{32113} one can obtain
\begin{equation}\label{32155}
\begin{split}
F^{ii}v_{ii1}=e^{(p-1+\varepsilon+k-q)v}(1+|\nabla v|^{2})^{\frac{k+1-q}{2}}(\varphi_{1}+\varepsilon v_{1}\varphi).
\end{split}
\end{equation}
Substituting \eqref{32154} and \eqref{32155} into \eqref{32153} yields
\begin{equation*}\label{32156}
\begin{split}
0\geq e^{(p-1+\varepsilon+k-q)v}(1+|\nabla v|^{2})^{\frac{k+1-q}{2}}(\varphi_{1}+\varepsilon v_{1}\varphi)+Cv_{1}[e^{(p+\varepsilon+k-q)v}(1+|\nabla v|^{2})^{\frac{k+1-q}{2}})\varphi]^{\frac{k-l-1}{k-l}},
\end{split}
\end{equation*}
which implies
\begin{equation*}
\begin{split}
0\geq (1+|\nabla v|^{2})^{\frac{k+1-q}{2(k-l)}}\varphi_{1}+Cv_{1}e^{-\frac{\varepsilon v}{k-l}}\varphi^{\frac{k-l-1}{k-l}}.
\end{split}
\end{equation*}
By Theorem \ref{L2}, $|e^{\varepsilon v}|=|u^\varepsilon|$ is uniformly bounded independently of $\varepsilon$. Hence there exists a positive constant $C_{1}$ depending on $n, k, l, \mathscr{P}, \min_{\mathbb{S}^{n}} \varphi$, and $\|\varphi\|_{C^{1}}$, but independent of $\varepsilon$, such that
\begin{equation*}
\begin{split}
0\geq C_{1}v_{1}-C(1+|\nabla v|^{2})^{\frac{k+1-q}{2(k-l)}},
\end{split}
\end{equation*}
since $p=q-l>1$, we obtain
\begin{equation*}
\begin{split}
|\nabla v(x_{0})| \leq C.
\end{split}
\end{equation*}
\end {proof}
\subsection{$C^{2}$ estimate}\,\\
\begin {theorem}\label{32159}
Let $0\leq l<k\leq C_{n-1}^{\mathscr{P}-1}$, $n\geq2$, $\mathscr{P}>1$, and $p=q-l>1$. Let $\varphi\in C^{2}(S^{n})$ be a positive function. Suppose that $u\in C^{4}(\mathbb{S}^{n})$ is a positive admissible solution of the equation \eqref{32142}. Define $$\widetilde{u}=\frac{u}{\min_{\mathbb{S}^{n}}u}.$$
Then there exists a positive constant $C''$, depending only on $n, k, l, \mathscr{P}, \min_{\mathbb{S}^{n}} \varphi$ and $\parallel\varphi\parallel_{C^{2}}$ such that
\begin{equation*}
\begin{split}
|\nabla^{2}\widetilde{u}(x)|\leq C'',\quad  \forall x\in \mathbb{S}^{n}.
\end{split}
\end{equation*}
\end {theorem}

\begin {proof}
We see that $\widetilde{u}$ satisfies
\begin{equation}\label{32144}
\begin{split}
\frac{\sigma_{k}(\Lambda(\nabla^{2}\widetilde{u}+\widetilde{u}I))}{\sigma_{l}(\Lambda(\nabla^{2}\widetilde{u}+\widetilde{u}I))}=
\widetilde{u}^{p-1+\varepsilon}(\widetilde{u}^{2}+|\nabla \widetilde{u}|^{2})^{\frac{k+1-q}{2}}(\min_{\mathbb{S}^{n}}u)^{\varepsilon}\varphi(x),\quad x\in \mathbb{S}^{n}.
\end{split}
\end{equation}
According to Theorem \ref{L2}
\begin{equation*}
\begin{split}
\frac{C_{N}^{k}}{C_{N}^{l}}\frac{\mathscr{P}^{k-l}}{\max\varphi}\leq (\min _{\mathbb{S}^{n}}u)^{\varepsilon}
\leq\frac{C_{N}^{k}}{C_{N}^{l}}\frac{\mathscr{P}^{k-l}}{\min\varphi}.
\end{split}
\end{equation*}
Moreover, by Theorem \ref{L3}, there exist positive constants $C$ and $C^{\prime}$ depending on $n$, $k$, $l$, $\mathscr{P}$, $\min_{\mathbb{S}^{n}} \varphi$, and $|\varphi|_{C^{1}}$, but independent of $\varepsilon$ such that
\begin{equation}\label{32145}
\begin{split}
1\leq\widetilde{u}\leq\frac{\max_{\mathbb{S}^{n}}u(x)}{\min_{\mathbb{S}^{n}}u(x)}\leq C,
\end{split}
\end{equation}
and
\begin{equation}\label{32146}
\begin{split}
|\nabla \widetilde{u}(x)|=\frac{u(x)}{\min_{\mathbb{S}^{n}}u(x)}\frac{|\nabla u(x)|}{u(x)}\leq \frac{\max_{\mathbb{S}^{n}}u(x)}{\min_{\mathbb{S}^{n}}u(x)}\frac{|\nabla u(x)|}{u(x)}\leq C^{\prime}.
\end{split}
\end{equation}
Consider the equation \eqref{32144} together with \eqref{32145}-\eqref{32146},  Theorem \ref{L4} yields $$|\nabla^{2}\widetilde{u}(x)|\leq C'',$$
where $C''$ depends on $n$, $k$, $l$, $\mathscr{P}$, $\min_{\mathbb{S}^{n}} \varphi$, and $\parallel\varphi\parallel_{C^{2}}$, but independent of $\varepsilon$.
\end {proof}
Together with Theorem \ref{L10}, this yields the following theorem.
\begin {theorem}\label{32160}
Let $0\leq l<k\leq N$, $n\geq2$, $\mathscr{P}>1$, $p=q-l>1$, and $q=k+1$. Let $\varphi\in C^{2}(S^{n})$ be a positive function. Suppose that $u\in C^{4}(\mathbb{S}^{n})$ is a positive admissible solution of the equation \eqref{32142}. Define $$\widetilde{u}=\frac{u}{\min_{\mathbb{S}^{n}}u}.$$
Then there exists a positive constant $C''$, depending only on $n, k, l, \mathscr{P}, \min_{\mathbb{S}^{n}} \varphi$ and $\|\varphi\|_{C^{2}}$ such that
\begin{equation*}
\begin{split}
|\nabla^{2}\widetilde{u}(x)|\leq C'',\quad  \forall x\in \mathbb{S}^{n}.
\end{split}
\end{equation*}
\end {theorem}

\subsection{\textbf{Existence and uniqueness of the positive solutions}}
\begin {lemma}\label{iejro-1}
Let $0\leq l<k\leq N$, $n\geq2$, $\mathscr{P}>1$, and $p=q-l$. Let $\varphi\in C^{0}(S^{n})$ be a positive function. Suppose that $u_{i}$, $i=1,2$, are positive admissible solutions of the equation
\begin{equation*}
\begin{split}
\frac{\sigma_{k}(\Lambda(\nabla^{2}u_{i}+u_{i}I))}{\sigma_{l}(\Lambda(\nabla^{2}u_{i}+u_{i}I))}=u_{i}^{p-1}(u_{i}^{2}+|\nabla u_{i}|^{2})^{\frac{k+1-q}{2}}\gamma_{i}\varphi(x),
\end{split}
\end{equation*}
where $\gamma_{i}$, $i=1,2$, are constants. Then $\gamma_{1}=\gamma_{2}$.
\end {lemma}
\begin {proof}
Suppose that $G=\frac{u_{1}}{u_{2}}$ attains its maximum at $x_{0}$. Then at $x_{0}$
\begin{equation*}
\begin{split}
0=\nabla \log G=\frac{\nabla u_{1}}{u_{1}}-\frac{\nabla u_{2}}{u_{2}},
\end{split}
\end{equation*}
and
\begin{equation*}
\begin{split}
0&\geq\nabla^{2}\log G\\
&=\frac{\nabla^{2} u_{1}}{u_{1}}-\frac{\nabla u_{1}\otimes \nabla u_{1}}{u_{1}^{2}}-\frac{\nabla^{2} u_{2}}{u_{2}}+\frac{\nabla u_{2}\otimes \nabla u_{2}}{u_{2}^{2}}\\
&=\frac{\nabla^{2} u_{1}}{u_{1}}-\frac{\nabla^{2} u_{2}}{u_{2}},
\end{split}
\end{equation*}
which implies  $u_{1}^{-1}(\nabla^{2}u_{1}+u_{1}I)\leq u_{2}^{-1}(\nabla^{2}u_{2}+u_{2}I)$. By direct calculations we obtain
\begin{equation*}
\begin{split}
\frac{\sigma_{k}(\Lambda(u_{1}^{-1}(\nabla^{2}u_{1}+u_{1}I)))}{\sigma_{l}(\Lambda(u_{1}^{-1}(\nabla^{2}u_{1}+u_{1}I)))}\leq \frac{\sigma_{k}(\Lambda(u_{2}^{-1}(\nabla^{2}u_{2}+u_{2}I)))}{\sigma_{l}(\Lambda(u_{2}^{-1}(\nabla^{2}u_{2}+u_{2}I)))}.
\end{split}
\end{equation*}
Therefore
\begin{equation*}
\begin{split}
\frac{\gamma_{1}}{\gamma_{2}}&=\frac{\gamma_{1}\varphi(x_{0})}{\gamma_{2}\varphi(x_{0})}=\frac{u_{1}^{1-p}(u_{1}^{2}+|\nabla u_{1}|^{2})^{\frac{q-1-k}{2}}\frac{\sigma_{k}(\Lambda(\nabla^{2}u_{1}+u_{1}I))}{\sigma_{l}(\Lambda(\nabla^{2}u_{1}+u_{1}I))}}{u_{2}^{1-p}(u_{2}^{2}+|\nabla u_{2}|^{2})^{\frac{q-1-k}{2}}\frac{\sigma_{k}(\Lambda(\nabla^{2}u_{2}+u_{2}I))}{\sigma_{l}(\Lambda(\nabla^{2}u_{2}+u_{2}I))}}\leq 1.
\end{split}
\end{equation*}
 The treatment for $\frac{\gamma_{1}}{\gamma_{2}}\geq1$ is similar. Therefore $\gamma_{1}\equiv \gamma_{2}$.
\end {proof}

\begin {lemma}\label{iuerte}
Let $0\leq l<k\leq N$, $n\geq2$, $\mathscr{P}>1$, and $p=q-l$. Let $\varphi\in C^{0}(S^{n})$ be a positive function. Suppose that $u$ and $\widehat{u}$ are positive admissible solutions of the equation
\begin{equation*}
\begin{split}
\frac{\sigma_{k}(\Lambda(\nabla^{2}u+uI))}{\sigma_{l}(\Lambda(\nabla^{2}u+uI))}=u^{p-1}(u^{2}+|\nabla u|^{2})^{\frac{k+1-q}{2}}\gamma\varphi(x),
\end{split}
\end{equation*}
with the same constant $\gamma$. Then $u=\widehat{u}$.
\end {lemma}
\begin {proof}
Let
$$M(u)=\frac{\sigma_{k}(\Lambda(\nabla^{2}u+uI))/\sigma_{l}(\Lambda(\nabla^{2}u+uI))}{u^{p-1}(u^{2}+|\nabla u|^{2})^{\frac{k+1-q}{2}}}.$$
Then
$$M(u)-M(\widehat{u})=\gamma\varphi-\gamma\varphi=0.$$

Since $M$ is invariant under scaling, we may assume $u\leq\widehat{u}$ and $u(x_{0})=\widehat{u}(x_{0})$ for some point $x_{0}\in\mathbb{S}^{n}$. Denote
\begin{equation*}
\begin{split}
u_{t}=tu+(1-t)\widehat{u},\quad \forall\,0\leq t\leq1,
\end{split}
\end{equation*}
\begin{equation*}
\begin{split}
a_{ij}(t)=t(\nabla u^{2}+uI)+(1-t)(\nabla\widehat{u}^{2}+\widehat{u}I),
\end{split}
\end{equation*}
then
\begin{equation*}
\begin{split}
0&=M(u)-M(\widetilde{u})=\int^{1}_{0}\frac{d}{dt}M(u_{t})dt\\
&=\sum_{i,j=1}^{n}b_{ij}(x)(u-\widehat{u})_{ij}+\sum_{i=1}^{n}c_{i}(x)(u-\widehat{u})_{i}+d(x)(u-\widehat{u}),
\end{split}
\end{equation*}
where
\begin{equation*}
\begin{split}
b_{ij}=\int^{1}_{0}u_{t}^{1-p}(u_{t}+|\nabla u_{t}|^{2})^{-\frac{k+1-q}{2}}\frac{\partial[\sigma_{k}(\Lambda_{t})/\sigma_{l}(\Lambda_{t})]}{\partial[a_{ij}(t)]}dt,
\end{split}
\end{equation*}
\begin{equation*}
\begin{split}
c_{i}=-(k+1-q)\int^{1}_{0}u_{t}^{1-p}(u_{t}+|\nabla u_{t}|^{2})^{-\frac{-k-3+q}{2}}\sum_{i=1}^{n}(tu_{i}+(1-t)\widetilde{u}_{i})\frac{\sigma_{k}(\Lambda_{t})}{\sigma_{l}(\Lambda_{t})}dt,
\end{split}
\end{equation*}
\begin{equation*}
\begin{split}
d=\int^{1}_{0}u_{t}^{-p}(u_{t}+|\nabla u_{t}|^{2})^{-\frac{k+1-q}{2}}\bigg(u_{t}&\sum_{i=1}^{n}\frac{\partial[\sigma_{k}(\Lambda_{t})/\sigma_{l}
(\Lambda_{t})]}{\partial[a_{ii}(t)]}\\
&-(p-1)\frac{\sigma_{k}(\Lambda_{t})}{\sigma_{l}(\Lambda_{t})}
-\frac{(k+1-p)u_{t}^{2}}{u_{t}^{2}+|\nabla u_{t}|^{2}}\frac{\sigma_{k}(\Lambda_{t})}{\sigma_{l}(\Lambda_{t})}\bigg)dt.
\end{split}
\end{equation*}

It follows from the maximum principle that $u-\widehat{u}\equiv 0$ on $\mathbb{S}^{n}$. Hence, the solution is unique.
\end {proof}
\begin {proof}[\textbf{{Proof of Theorem \ref{L7}}.}]
The uniqueness results can be easily obtained by Lemma \ref{iejro-1} and \ref{iuerte}. In the following, we prove the existence part by approximation.

For any small $\varepsilon$, the equation \eqref{32142} has a unique positive admissible solution $u_{\varepsilon}$. Denote $\widetilde{u}_{\varepsilon}=\frac{u_{\varepsilon}}{\min_{\mathbb{S}^{n}}u_{\varepsilon}}$. Then $\widetilde{u}_{\varepsilon}$ satisfies
\begin{equation*}
\begin{split}
\frac{\sigma_{k}(\nabla^{2}\widetilde{u}_{\varepsilon}+\widetilde{u}_{\varepsilon}I)}
{\sigma_{l}(\nabla^{2}\widetilde{u}_{\varepsilon}+\widetilde{u}_{\varepsilon}I)}
=\widetilde{u}_{\varepsilon}^{p-1+\varepsilon}(\widetilde{u}_{\varepsilon}^{2}+|\nabla \widetilde{u}_{\varepsilon}|^{2})
^{\frac{k+1-q}{2}}(\min_{\mathbb{S}^{n}}\widetilde{u}_{\varepsilon})^{\varepsilon}\varphi(x).
\end{split}
\end{equation*}
Letting $\varepsilon \rightarrow 0^{+}$, we have $|\nabla (\widetilde{u}_{\varepsilon})^{\varepsilon}|=\varepsilon(\widetilde{u}_{\varepsilon})^{\varepsilon-1}|\nabla (\widetilde{u}_{\varepsilon})|\rightarrow 0$ by \eqref{32145}-\eqref{32146}. Then $(\min_{\mathbb{S}^{n}}\widetilde{u}_{\varepsilon})^{\varepsilon}$ converges to a positive constant $\gamma$. Thus $\widetilde{u}_{\varepsilon}$ converges to a positive admissible solution $u$ of equation $$\frac{\sigma_{k}(\Lambda)}{\sigma_{l}(\Lambda)}=u^{p-1}(u^{2}+|\nabla u|^{2})^{\frac{k+1-q}{2}}\gamma\varphi(x).$$
Moreover, theorem \ref{Full-rank-1} ensures the strict spherical convexity of $u$.
\end {proof}
\begin {proof}[\textbf{{Proof of Theorem \ref{L8}} (\romannumeral2).}]
The proof is analogous to that in Theorem \ref{L7} and so is omitted. It is important to note that Theorem \ref{32160} is applied to establish the convergence of $\widetilde{u}_{\varepsilon}$ to a strictly spherically convex function $u$.
\end {proof}

%%%%%%%%%%%%%%%%%%%%%%%%%%%%%%%%%%%%%%%%%%%%%%%%%%%%%%%%%%%%%%%%%%%%%%%%%%%%%%%%%%%%%%%%%%%%%%%%%%%%%%%%%%%%%%%%%%%%%%%%%%%%%%%%%%%%%%%%%%%%%%%%%%%%%%%%%%%%%%%%%%%%


\begin{thebibliography}{50}
\setlength{\itemsep}{-0pt}
\small
\bibitem{Ale-38}
A. Aleksandrov,
 On the theory of mixed volumes. III. Extensions of two theorems of Minkowski on convex polyhedra to arbitrary convex bodies,
 Mat. Sb. (N.S.)
 3,
 27-46,
 (1938)%少了number

\bibitem{Ale-39}
A. Aleksandrov,
 On the surface area measure of convex bodies,
 Mat. Sb. (N.S.)
 6,
 167-174,
 (1939)%少了number

\bibitem{BP}
B.Bian, P.Guan,
 A microscopic convexity principle for nonlinear partial differential equations,
 Invent. Math.,
 177,
 203-335,
 (2009)%少了number

\bibitem{Ber-69}
C. Berg,
 Corps convexes et potentiels sph\'eriques,
 Mat.-Fys. Medd. Danske Vid. Selsk.,
 37(6),
 64,
 (1969)

\bibitem{BF-2019}
K. B\"or\"oczky, F. Fodor,
 The $L_p$ dual Minkowski problem for $p>1$ and $q>0$,
 J. Differential Equations,
 266(12),
 7980-8033,
 (2019)

\bibitem{Bianchi2019}
G. Bianchi, K. B\"or\"oczky, A. Colesanti, D. Yang,
 The $L_{p}$-Minkowski problem for $ -n< p < 1 $,
 Adv. Math.,
 341,
 493-535,
 (2019)%少了number

\bibitem{BIS-2023}
P. Bryan, M. Ivaki, J. Scheuer,
 Christoffel-Minkowski flows, Trans.
 Amer. Math. Soc.,
 376(4),
 2373-2393,
 (2023)

\bibitem{Boreczky2013}
K. B\"or\"oczky, E. Lutwak, D. Yang, G. Zhang,
 The logarithmic Minkowski problem,
 J. Amer. Math. Soc.,
 26(3),
 831-852,
 (2013)

\bibitem{Bereczky2017}
K. B\"or\"oczky, H. T. Trinh,
 The planar $L_{p}$-Minkowski problem for $0<p<1$,
 Adv. in Appl. Math.,
 87,
 58-81,
 (2017)%少了number

\bibitem{CNS-85}
L. Caffarelli, L. Nirenberg, J. Spruck,
 The Dirichlet problem for nonlinear second order elliptic equations, III: Functions of the eigenvalues of the Hessian,
 Acta Math.,
 155,
 261-301,
 (1985)%少了number

\bibitem{CH-25}
C. Cabezas-Moreno, J. Hu,
 The  $Lp$  dual Christoffel-Minkowski problem for  $1<p<q\leq k+1$  with  $1\leq k \leq n$,
 Calc. Var. Partial Differential Equations,
 64(7),
 229,
 (2025)

\bibitem{CC}
C. Chen,
 On the elementary symmetric functions, preprint.

\bibitem{CDH}
C. Chen, W. Dong, F. Han,
 Interior Hessian estimates for a class of Hessian type equations,
 Calc. Var. Partial Differential Equations,
 62(52),
 1-15,
 (2023)

\bibitem{CX-22}
C. Chen, L. Xu,
 The $L^p$ Minkowski type problem for a class of mixed Hessian quotient equations,
 Adv. Math.,
 411,
 108794,
 (2022)%少了number

\bibitem{CX-23}
C. Chen, L. Xu,
 Uniqueness of solutions to a class of mixed Hessian quotient type equations,
 J. Geom. Anal.,
 33(210),
 18,
 (2023)

\bibitem{CHZ-2019}
C. Chen, Y. Huang, Y. Zhao,
 Smooth solutions to the $L_p$ dual Minkowski problem,
 Math. Ann.,
 373(3-4),
 953-976,
 (2019)

\bibitem{CTX-20}
X. Chen, Q. Tu, N. Xiang,
 A class of Hessian quotient equations in Euclidean space,
 J. Differential Equations,
 269(2020),
 11172-11194,
 (2020)


\bibitem{CTX3}
X. Chen, Q. Tu, N. Xiang,
 The Dirichlet problem for a class of Hessian quotient equations on Riemannian manifolds,
 Int. Math. Res. Not.,
 2023(12),
 10013-10036,
 (2023)

\bibitem{CTX-25}
X. Chen, Q. Tu, N. Xiang,
 The $L_p$ dual Christoffel-Minkowski problem for the case $p\geq q$,
 arXiv: 2503.01454
 (2025)%有大问题不会改

\bibitem{CTX-21}
L. Chen,  Q. Tu, N. Xiang,
 Pogorelov type estimates for a class of Hessian quotient equations,
 J. Differential Equations,
 282,
 272-284,
 (2021)%少了number

\bibitem{Chen2017}
S. Chen, Q. Li, G. Zhu,
 On the $L_{p}$ Monge-Amp\'ere equation,
 J. Differential Equations,
 263(8),
 4997-5011,
 (2017)

\bibitem{Chen-Li-2021}
H. Chen, Q. Li,
 The $L_p$  dual Minkowski problem and related parabolic flows,
 J. Funct. Anal.,
 281(8),
 109139,
 (2021)

\bibitem{Ch-Yau-76}
S. Cheng, S. Yau,
 On the regularity of the solution of the $n$-dimensionalMinkowski problem,
 Comm. Pure Appl. Math.,
 29(5),
 495-516,
 (1976)

\bibitem{Cho-65}
E. Christoffel,
 Ueber die Bestimmung der Gestalt einer krummen Oberflche durch lokale Messungenauf derselben,
 J. Reine Angew. Math.,
 64,
 193-209,
 (1865)%少了number

\bibitem{Chou2006}
K. Chou, X. Wang,
 The $L_{p}$-Minkowski problem and the Minkowski problem in centroaffine geometry,
 Adv. Math.,
 205,
 33-83,
 (2006)%少了number

\bibitem{CJ-21}
J. Chu, H. Jiao,
 Curvature estimates for a class of Hessian type equations,
 Calc. Var. Partial Differential Equations,
 60(90),
 1-18,
 (2021)

\bibitem{Dong-23}
W. Dong,
 Curvature estimates for p-convex hypersurfaces of prescribed curvature,
 Rev. Mat. Iberoam.,
 39,
 1039-1058,
 (2023)%少了number

\bibitem{Dong-24}
W. Dong,
 The Dirichlet problem for prescribed curvature equations of  p-convex hypersurfaces,
 Manuscripta Math.,
 174,
 785-806,
 (2024)%少了number

\bibitem{D-23}
S. Dinew,
 Interior estimates for  p-plurisubharmonic functions,
 Indiana Univ. Math. J.,
 72(5),
 2025-2057,
 (2023)

\bibitem{DL-2023}
S. Ding, G. Li,
 A class of inverse curvature flows and $L_p$ dual Christoffel-Minkowski problem,
 Trans. Amer. Math. Soc.,
 376(1),
 697-752,
 (2023)

\bibitem{Fir-67}
W. Firey,
 The determination of convex bodies from their mean radius of curvature functions,
 Mathematika,
 14,
 1-13,
 (1967) %少了number

\bibitem{Fir-68}
W. Firey,
 Christoffel's problem for general convex bodies,
 Mathematika,
 15,
 7-21,
 (1968) %少了number

\bibitem{Fir-70}
W. Firey,
 Intermediate Christoffel-Minkowski problems for figures of revolution,
 Israel J. Math.,
 8,
 384-390,
 (1970)%少了number

\bibitem{G-84}
P. Gauduchon,
 La 1-forme de torsion d'une variete hermitienne compacte,
 Math. Ann.,
 267,
 495-518,
 (1984) %少了number

\bibitem{Gerhardt2006}
C. Gerhardt,
 Curvature problems, Series in Geometry and Topology,
 International Press of Boston Inc. Sommerville,
 39,
 (2006) %有大毛病不会改

\bibitem{GN-21}
B. Guan, X. Nie,
 Second order estimates for fully nonlinear elliptic equations with gradient terms on Hermitian manifolds,
 arXiv :2108 .03308%有大毛病不会改

\bibitem{GZT-25}
J. Gong, Z. Liu, Q. Tu,
 The Neumann problem for a class of Hessian quotient type equations,
 J. Differential Equations,
 431,
 113251,
 (2025) %少了number

\bibitem{P}
P. Guan,
 Topics in geometric fully nonlinear equations,
 Lecture Notes,
 (2002)%s少了好些东西

\bibitem{GM-2003}
P. Guan, X. Ma,
 The Christoffel-Minkowski problem. I. Convexity of solutions of aHessian equation,
 Invent. Math,
 151(3),
 553-577,
 (2003)

\bibitem{GL-99}
P. Guan, C. Lin,
 On equation $\det(u_{ij} + \delta_{ij} u) = u^p f$ on $S^n$,
 manuscript,
 (1999) %有大毛病不会改

\bibitem{GLM-2006}
P. Guan, C. Lin, X. Ma,
 The Christoffel-Minkowski problem. II.Weingarten curvature equations,
 Chinese Ann. Math. Ser. B,
 27(6),
 595-614,
 (2006)

\bibitem{GMZ-2006}
 P. Guan, X. Ma,  F. Zhou,
 The Christofel-Minkowski problem. III. Existence and convexity of admissible solutions,
 Comm. Pure Appl. Math.,
 59(9),
 1352-1376,
 (2006)

\bibitem{GX-18}
P. Guan, C. Xia,
 $L^p$ Christoffel-Minkowski problem: the case $1 < p < k+1$,
 Calc. Var. Partial Differential Equations,
 57 (2018): 69%有大毛病不会改

\bibitem{GCF1999}
Gauss curvature flow,
 the fate of the rolling stones,
 Invent. Math.,
 138,
 151-161,
 (1999)%少了numkber

\bibitem{HL-13}
F. Harvey, H. Lawson,
 p-convexity, p-plurisubharmonicity and the Levi problem,
 Indiana Univ. Math. J.,
 62,
 149-169,
 (2013)%少了numkber（查不到）

\bibitem{HMS-04}
C. Hu, X. Ma, C. Shen,
 On the Christoffel-Minkowski problem of Firey's $p$-sum,
 Calc. Var. Partial Differential Equations,
 21(2),
 137-155,
 (2004)

\bibitem{HZ-2018}
Y. Huang, Y. Zhao,
 On the $L_p$ dual Minkowski problem,
 Adv. Math.,
 332,
 57-84,
 (2018)%少了numkber（查不到）

\bibitem{IVAKI-2019}
M. Ivaki,
 Deforming a hypersurface by principal radii of curvature and support function,
 Calc. Var. Partial Differential Equations,
 58(1),
 1,
 (2019)

\bibitem{Lieberman1996}
G. Lieberman,
 Second order parabolic differential equations,
 World Scientific,
 (1996)%缺斤少两（查不到）

\bibitem{Lu2013}
J. Lu, X. Wang,
 Rotationally symmetric solutions to the $L_{p}$-Minkowski problem,
 J. Differential Equations,
 254(3),
 983-1005,
 (2013)

\bibitem{Lut-1993}
E. Lutwak,
 The Brunn-Minkowski-Firey theory. I. Mixed volumes and the Minkowski problem,
 J. Differential Geom.,
 38(1),
 131-150,
 (1993)

\bibitem{Lutwak2004}
E. Lutwak, D. Yang, G. Zhang,
 On the $L_{p}$-Minkowski problem, Trans.
 Amer. Math. Soc.,
 356(11),
 4359-4370,
 (2004)

\bibitem{LYZ-2018-1}
E. Lutwak, D. Yang, G. Zhang,
 $L_p$ dual curvature measures,
 Adv. Math.,
 329,
 85-132,
 (2018)%少了numkber

\bibitem{Min-97}
H. Minkowski,
 Allgemeine Lehrs\"atze \"uber die convexen Polyeder,
 Nachr. Ges. Wiss. Gttingen,
 198-219,
 (1897)%少了numkber和volume（查不到）

\bibitem{Min-03}
H. Minkowski,
 Volumen und Oberfl\"ache,
 Math. Ann.,
 57(4),
 447-495,
 (1903)

\bibitem{Nir-53}
L. Nirenberg,
 The Weyl and Minkowski problems in differential geometry in the large,
 Comm. Pure Appl. Math.,
 6(3),
 337-394,
 (1953)

\bibitem{Pog-78}
A. Pogorelov,
 The Minkowski multidimensional problem,
 translated from the Russian by Vladimir Oliker,Scripta Series in Mathematics, Winston, Washington, DC %真不会改
 (1978)

\bibitem{ShC-2020}
W. Sheng, C. Yi,
 A class of anisotropic expanding curvature flows,
 Discrete Contin. Dyn. Syst.,
 40(4),
 2017-2035,
 (2020)

\bibitem{STW-17}
G. Sz\'ekelyhidi, V. Tosatti, B. Weinkove,
 Gauduchon metrics with prescribed volume form,
 Acta Math.,
 219,
 181-211,
 (2017)

\bibitem{TW-17}
V. Tosatti, B. Weinkove,
 The Monge-Amp\'ere equation for $(n-1)$-plurisubharmonic functions on a compact K\"ahler manifold,
 J. Amer. Math. Soc.,
 30(2),
 311-346,
 (2017)

\bibitem{Zhu2014}
G. Zhu,
 The logarithmic Minkowski problem for polytopes,
 Adv. Math.,
 262,
 909-931,
 (2014)%少了numkber（查不到）

\bibitem{Zhu2015a}
G. Zhu,
 The $L_{p}$-Minkowski problem for polytopes for $0 < p < 1$,
 J. Funct. Anal.,
 269(4),
 1070-1094,
 (2015)

\bibitem{Zhu2015b}
G. Zhu,
 The centro-affine Minkowski problem for polytopes,
 J. Differential Geom.,
 101(1),
 159-174,
 (2015)

\bibitem{Z-24}
J. Zhou,
 Curvature estimates for a class of Hessian quotient type curvature equations,
 Calc. Var. Partial Differential Equations,
 63(4),
 88,
 (2024)



\end{thebibliography}
\end{document}